\documentclass[12pt,draft]{article}

\usepackage{amsmath}
\usepackage{amssymb}
\usepackage{amsthm}
\usepackage{geometry}
\usepackage{titlesec}
\usepackage[numbers,sort&compress]{natbib}
\usepackage{tabularx}
\usepackage{booktabs}

\numberwithin{equation}{section}
\newtheorem{theorem}{Theorem}[section]
\newtheorem{lemma}[theorem]{Lemma}
\newtheorem{proposition}[theorem]{Proposition}
\newtheorem{corollary}[theorem]{Corollary}
\theoremstyle{definition}
\newtheorem{definition}[theorem]{Definition}
\newtheorem{remark}[theorem]{Remark}
\theoremstyle{plain}
\newtheorem{example}[theorem]{Example}
\titleformat*{\section}{\large\bfseries}
\titleformat*{\subsection}{\normalsize\bfseries}

\newcommand{\bC}{\ensuremath{\mathbb{C}}}
\newcommand{\bCx}{\ensuremath{\mathbb{C}^*}}
\newcommand{\bZ}{\ensuremath{\mathbb{Z}}}
\newcommand{\subno}[1]{\noindent\textnormal{(#1)}}
\newcommand{\afCase}[1]{\textbf{Case~#1:}}

\newcommand{\bZz}{\ensuremath{\mathbb{Z}_2}}
\newcommand{\ep}{\ensuremath{{\bar{0}}}}
\newcommand{\op}{\ensuremath{{\bar{1}}}}
\newcommand{\cp}{\ensuremath{\mathbb{C}[\partial]}}
\newcommand{\hv}{\ensuremath{\mathcal{HV}}}
\newcommand{\hvs}{\ensuremath{\mathcal{HVS}}}
\newcommand{\hvsi}{\ensuremath{\mathcal{HVS}(\alpha)}}
\newcommand{\hvsii}{\ensuremath{\mathcal{HVS}(\beta,\gamma,\tau)}}
\newcommand{\Chom}{\mathrm{Chom }}
\newcommand{\Cend}{\mathrm{Cend }}
\newcommand{\CDer}{\mathrm{CDer }}
\newcommand{\CInn}{\mathrm{CInn }}
\newcommand{\ad}{\mathrm{ad\ }}
\newcommand{\dl}{\ensuremath{d_\lambda }}
\newcommand{\Aut}{\mathrm{Aut }}
\newcommand{\BDer}{\mathrm{BDer }}
\newcommand{\BInn}{\mathrm{BInn }}
\newcommand{\bp}{\ensuremath{\bar{\partial}}}
\newcommand{\bl}{\ensuremath{\bar{\lambda}}}

\begin{document}

\begin{center}
  \bfseries\large
  Two classes of Lie conformal superalgebras related to the Heisenberg--Virasoro Lie conformal algebra
  \footnote{\indent \ \ Corresponding author: Xiaoqing Yue (xiaoqingyue@tongji.edu.cn)}

    \mdseries\normalsize
    \bigskip
    Jinrong Wang, Xiaoqing Yue

    \footnotesize
    \smallskip
    School of Mathematical Sciences, Tongji University, Shanghai 200092, China

    \smallskip
    E-mails: 2130926@tongji.edu.cn, xiaoqingyue@tongji.edu.cn
\end{center}
{
\footnotesize
    \noindent\textbf{Abstract}.
    In this paper, firstly we construct two classes of Lie conformal superalgebras denoted by $\hvsi$ and $\hvsii$, respectively, where $\alpha$ is an nonzero complex number and $\beta,\gamma,\tau$ are complex numbers. They are both of rank $(2+1)$ and the even part $\hvsi_\ep=\hvsii_\ep$ is the Heisenberg--Virasoro Lie conformal algebra, which is a free $\bC[\partial]$-module generated by $L$ and $H$ satisfying $[L_\lambda L]=(\partial+2\lambda)L,\ [L_\lambda H]=(\partial+\lambda)H,\ [H_\lambda H]=0$. Then we completely determine the conformal derivations, conformal biderivations, automorphism groups and conformal modules of rank $(1+1)$ for these two Lie conformal superalgebras.

    \smallskip
    \noindent\textit{Keywords}:
    Lie conformal superalgebra, conformal derivation, conformal biderivation, automorphism group, conformal module

    \smallskip
    \noindent\textit{Mathematics Subject Classification (2020)}: 17B10, 17B40, 17B65, 17B68
}
\section{Introduction}\label{sect:intro}
\hspace{1.5em}Lie conformal superalgebras, introduced by Kac in \cite{K98}, encode the singular part of the operator product expansion of chiral fields in two-dimensional quantum field theory. And they play a significant role in the representation theory of infinite-dimensional Lie algebras \cite{K99}. As a generalization of Lie conformal algebras, the study of Lie conformal superalgebras has received a lot of attention in recent years. The classification of finite simple Lie conformal superalgebras and their representation theory have been discussed in \cite{BKL10,BKL13,BKLR,FK,CK}. The structure theory of finite Lie conformal superalgebras was studied in \cite{FKR}. Semidirect products, cohomology and generalized derivations of a Lie conformal superalgebra were discussed in \cite{ZCY}. The $\cp$-split extending structures problem of Lie conformal algebras was generalized to the super case in \cite{ZCY19}.

However, the structure theory and representation theory for non-simple Lie conformal superalgebras is in general more difficult. Even for the non-super cases, there is still a lot of work to be done. The simplest nontrivial Lie conformal algebra is the Virasoro Lie conformal algebra $\mathrm{Vir}=\bC[\partial]L$, where $L$ is the $\bC[\partial]$-generator with the relation $[L_\lambda L]=(\partial+2\lambda)L$. Since a simple Lie conformal algebra of rank one is isomorphic to the Virasoro Lie  conformal algebra $\mathrm{Vir}$ \cite{DK}, it is very natural to consider the problem for new non-simple Lie conformal (super)algebras containing $\mathrm{Vir}$.
The Heisenberg--Virasoro Lie conformal algebra $\hv$ is of rank 2 and has a free $\bC[\partial]$-basis $\{L,H\}$ such that:
    \begin{align}\label{HV}
        [L_\lambda L]=(\partial+2\lambda)L,\ [L_\lambda H]=(\partial+\lambda)H,\ [H_\lambda H]=0,
    \end{align}
which is an important Lie conformal algebra containing $\mathrm{Vir}$ as its subalgebra.
It was firstly introduced in \cite{SY} as a Lie conformal algebra associated with the twisted Heisenberg--Virasoro Lie algebra. We can see that $\hv$ has a nontrivial abelian conformal ideal with one free generator $H$ as a $\cp$-module. Thus it is neither simple nor semi-simple. Its conformal derivations, free nontrivial rank one conformal modules and universal central extensions were determined in \cite{SY}. The cohomology of the Heisenberg--Virasoro Lie conformal algebra with coefficients in its modules was computed in \cite{YW}. 

For the close relation between Lie conformal superalgebra and infinite-dimensional Lie superalgebras \cite{CK}, it is necessary to generalize the Heisenberg--Virasoro Lie conformal algebra to the super case. In \cite{CDH}, the super Heisenberg--Virasoro algebra was defined by using a class of Heisenberg--Virasoro Lie conformal modules. Representations of super deformation of Heisenberg--Virasoro type Lie conformal algebras were studied in \cite{WWX}. As a special case of infinite Lie conformal superalgebras of Block type, the finite irreducible conformal modules of Heisenberg--Neveu--Schwarz conformal algebra were classified in \cite{X19}. We hope that our work here will shed some light on the development of the structure theory and representation theory of Lie conformal superalgebras. Moreover, we believe that our results are also beneficial for understanding the Heisenberg type Lie superalgebras.

Motivated by the aforementioned facts, in this paper, we study the Lie conformal superalgebras $R=R_\ep\oplus R_\op$, where $R_\ep=\hv$ and $R_\op$ is of rank 1, and give the classification of them completely. Based on this, we construct two Lie conformal superalgebras $\hvsi$ and $\hvsii$.

Throughout this paper, we use notations $\bC$, $\bCx$ and $\bZ$ to represent the set of complex numbers, nonzero complex numbers and integers, respectively. In addition, all vector spaces and tensor products are over $\bC$. For any vector space $V$, we use $V[\lambda]$ to denote the set of polynomials of $\lambda$ with coefficients in $V$.
\begin{definition}\label{hsvi.def}
    The Lie conformal superalgebra $\hvsi=\cp L\oplus\cp H\oplus\cp G$ with one parameter $\alpha\in\bCx$ is a free \bZz-graded \cp-module, satisfying the following relations:
    \begin{align*}
        [L_\lambda L]=(\partial+2\lambda)L,\ [L_\lambda H]=&(\partial+\lambda)H,\ [L_\lambda G]=(\partial+\lambda)G,\\
        [G_\lambda G]=\alpha H,&\ [H_\lambda H]=[H_\lambda G]=0,
    \end{align*}
    where $\hvsi_\ep=\cp L\oplus\cp H$ and $\hvsi_\op=\cp G$.
\end{definition}

\begin{definition}\label{hsvii.def}
    The Lie conformal superalgebra $\hvsii=\cp L\oplus\cp H\oplus\cp E$ with three parameters $\beta, \gamma, \tau\in\bC$ is a free \bZz-graded \cp-module, satisfying the following relations:
    \begin{align*}
        [L_\lambda L]=(\partial+2\lambda)L,\ [L_\lambda H]=&(\partial+\lambda)H,\ [L_\lambda E]=(\partial+\beta\lambda+\gamma)E,\\
        [H_\lambda E]=\tau E,&\ [H_\lambda H]=[E_\lambda E]=0,
    \end{align*}
    where $\hvsii_\ep=\cp L\oplus\cp H$ and $\hvsii_\op=\cp E$.
\end{definition}
 Then, their conformal derivations and conformal biderivations are determined. We find that $\hvsi$ contains non-inner conformal derivations $\{\dl\in\CDer(\hvs)_\ep\ |\ \dl L=cH,\ \dl H=\dl G=0,\ \forall\ c\in\bCx\}$. If $\tau=0$, then $\hvsii$ contains non-inner conformal derivations $\{\dl\in\CDer\big(\hvsii\big)_\ep\ |\ \dl L=cH,\ \dl H=\dl E=0,\ \forall\ c\in\bCx\}$. If $\beta=0$ and $\tau=1$ at the same time, then $\hvsii$ contains non-inner odd conformal derivations $\{\dl\in\CDer\big(\hvsii\big)_\op\ |\ \dl L=cE,\ \dl H=\dl E=0,\ \forall\ c\in\bCx\}$. The conformal biderivations on $\hvsi$ are all inner conformal biderivations, and $\hvsii$ exists non-inner conformal biderivations if and only if $\beta=2,\gamma=0,\tau=0$. As a byproduct, we also determine the conformal derivations and conformal biderivations on the two important Lie conformal superalgebras (e.g. Neveu--Schwarz Lie conformal superalgebra).

The automorphism groups of them are also determined. For the Lie conformal superalgebra $\hvsi$, we have $\Aut(\hvsi)\cong\bCx\ltimes(\bC\times\bC)$. For $\hvsii$, if $\tau=0$, then $\Aut\big(\hvsii\big)\cong(\bCx\times\bCx)\ltimes(\bC\times\bC)$. Otherwise, $\Aut\big(\hvsii\big)\cong\bCx$.

Finally, we classify the free conformal modules of rank $(1+1)$ over these two algebras respectively. The rank $(1+1)$ free conformal modules over $\hvsi$ or $\hvsii$ with $\tau=0$ either have the structure of finite nontrivial irreducible conformal module over the Neveu--Schwarz Lie conformal superalgebra (cf.\cite{CK}), or have the structure of a class conformal modules constructed in \cite{WWX}. Additionally, $\hvsii$ with $\tau\neq 0$ has no non-trivial free conformal modules of rank $(1+1)$. 

The rest of this paper is organized as follows. In Section~\ref{sect:prelim}, we recall some basic definitions, notations and related results about Lie conformal superalgebras. In Section~\ref{sec:Lie conformal salg}, we firstly classify the Lie conformal superalgebras whose the even part is the Heisenberg--Virasoro Lie conformal algebra and the odd part is of rank one. Then, two classes of Lie conformal superalgebras are constructed. In Section~\ref{sect:conformal derivation} and \ref{sect:conformal biderivation}, their conformal derivations and conformal biderivations are discussed, respectively. In Section~\ref{sect:automor}, we determine the automorphism groups of these two classes Lie conformal superalgebras. Furthermore in the last section, we study the conformal modules of rank $(1+1)$ over $\hvs$ and $\hvsii$, respectively. Our main results are summarized in Theorems~\ref{hvs.d}, \ref{hvsii.d}, \ref{thm:biderivation}, \ref{thm:automor}, \ref{thm:conmod.hvs} and \ref{thm:conmod.hvsii}.

\section{Preliminaries}\label{sect:prelim}
\hspace{1.5em}In this section, we will introduce some basic definitions, notations and related results about Lie conformal superalgebras.

Let $V$ be a superspace that is a \bZz-graded linear space with a direct sum $V=V_\ep \oplus V_\op$. If $x\in V_\theta$, $\theta\in\bZz=\{\ep,\op\}$, then we say $x$ is \textit{homogeneous} and of \textit{parity $\theta$}. The parity of a homogeneous element $x$ is denoted by $|x|$. Throughout what follows, if $|x|$ occurs in an expression, then it is assumed that $x$ is homogeneous and that the expression extends to the other elements by linearity.

\begin{definition}
    A \textit{Lie conformal superalgebra} $R=R_\ep \oplus R_\op$ is a \bZz-graded $\bC[\partial]$-module with a $\bC$-linear map $R\otimes R\rightarrow \bC[\lambda]\otimes R$, $a\otimes b\mapsto [a_\lambda b]$ called $\lambda$-bracket, and satisfying the following axioms $(a,b,c\in R)$:
    \begin{align}
        \text{(conformal sesquilinearity)}&\phantom{1234} [\partial a_\lambda b]=-\lambda[a_\lambda b],\ [a_\lambda \partial b]=(\partial+\lambda)[a_\lambda b],\label{conformal-sesqui}\\
        \text{(skew-symmetry)}&\phantom{1234} [a_\lambda b]=-(-1)^{|a||b|}[b_{-\lambda-\partial}a],\label{skew-symm}\\
        \text{(Jacobi
        	identity)}&\phantom{1234} [a_\lambda[b_\mu c]]=[[a_\lambda b]_{\lambda+\mu}c]+(-1)^{|a||b|}[b_\mu[a_\lambda c]].
    \end{align}
\end{definition}

For a Lie conformal superalgebra $R=R_\ep\oplus R_\op$ and for any $\theta,\xi\in\{\ep,\op\}$, $R_\theta R_\xi\subseteq R_{\theta+\xi}$. $R$ is called \textit{finite} if it is finitely generated over $\bC[\partial]$, otherwise it is called \emph{infinite}. For a finite Lie conformal superalgebra $R$, if it is further a free $\bC[\partial]$-module, then the \emph{rank} of $R$ is defined as the rank of the underlying free $\bC[\partial]$-module. Sometimes the rank is written as the sum of two numbers which means the sum of the rank of $R_\ep$ and the rank of $R_\op$.

\begin{example}\textup{(\cite{CK}) (Neveu--Schwarz Lie conformal superalgebra)} Let $\mathcal{NS}=\bC[\partial]L\oplus \bC[\partial]G$ be a free \bZz-graded $\bC[\partial]$-module. Define
\begin{align*}
    [L_\lambda L]=(\partial+2\lambda)L,\ [L_\lambda G]=(\partial+\frac32\lambda)G,\ [G_\lambda G]=2L,
\end{align*}
    where $\mathcal{NS}_\ep=\bC[\partial]L$ and $\mathcal{NS}_\op=\bC[\partial]G$. Then $\mathcal{NS}$ is a Lie conformal superalgebra of rank $(1+1)$. We call it the Neveu--Schwarz Lie conformal superalgebra.
\end{example}

The Lie conformal superalgebras of rank $(1+1)$ have been classified completely:
\begin{proposition}\label{rank1.1}\textup{(\cite{ZCY})}
Let $R=\cp x\oplus\cp y$ be a Lie conformal superalgebra that is a free \bZz-graded \cp-module, where $R_\ep=\cp x$ and $R_\op=\cp y$. Then $R$ is isomorphic to one of the following Lie conformal superalgebras\textup{:}

\subno{1} $R_1$\textup{:} $[x_\lambda x]=0$, $[x_\lambda y]=0$, $[y_\lambda y]=p(\partial)x$, $\forall\ p(\partial)\in\cp$,

\subno{2} $R_2$\textup{:} $[x_\lambda x]=0$, $[y_\lambda y]=0$, $[x_\lambda y]=q(\partial)y$, $\forall\ q(\partial)\in\cp$,

\subno{3} $R_3$\textup{:} $[x_\lambda x]=(\partial+2\lambda)x$, $[x_\lambda y]=0$, $[y_\lambda y]=0$,

\subno{4} $R_4$\textup{:} $[x_\lambda x]=(\partial+2\lambda)x$, $[x_\lambda y]=(\partial+\beta\lambda+\gamma)y$, $[y_\lambda y]=0$, $\forall\ \beta,\gamma\in\bC$,

\subno{5} $R_5$\textup{:} $[x_\lambda x]=(\partial+2\lambda)x$, $[x_\lambda y]=(\partial+\frac32\lambda)y$, $[y_\lambda y]=\alpha x$, $\forall\ \alpha\in\bCx$.
\end{proposition}
Obviously, the Neveu--Schwarz Lie conformal superalgebra is the special case of $R_5$.
\begin{definition}
    Let $V,W$ be two \bZz-graded \cp-modules, respectively. A \bC-linear map
    \begin{align*}
        \varphi: V\rightarrow\bC[\lambda]\otimes W
    \end{align*}
    is called a \textit{conformal linear map of degree $\theta$} if the following holds:
    \begin{align*}
        \varphi_\lambda(\partial v)=(\partial+\lambda)(\varphi_\lambda v)\text{ and }\varphi_\lambda(V_\xi)\subseteq W_{\theta+\xi},
    \end{align*}
    where $\theta,\xi\in\bZz$, for any $v\in V$.
\end{definition}
The set of conformal linear maps of degree $\theta$ from $V$ to $W$ is denoted by $\Chom(V,W)_\theta$. Then $\Chom(V,W)=\Chom(V,W)_\ep+\Chom(V,W)_\op$ is a \bZz-graded \cp-module via
    \begin{align*}
        (\partial\varphi)_\lambda v=-\lambda\varphi_\lambda v,\text{ for any }v\in V.
    \end{align*}
    In particular, $\Chom(V,V)$ are denoted by $\Cend(V)$.
\begin{definition}
    A \textit{conformal module} $M=M_\ep+M_\op$ over a Lie conformal superalgebra $R$ is a \bZz-graded \cp-module equipped with a \bC-linear map $R\otimes M\rightarrow\bC[\lambda]\otimes M$, $a\otimes v\mapsto a_\lambda v$ called $\lambda$-action, and satisfying the following relations $(a,b\in R, v\in M)$:
    \begin{align}
        (\partial a)_\lambda v=-\lambda a_\lambda v,\ a_\lambda(\partial v)=(\partial+\lambda)a_\lambda v,\label{conmod.1}\\
        [a_\lambda b]_{\lambda+\mu}v=a_\lambda(b_\mu v)-(-1)^{|a||b|}b_\mu(a_\lambda v).\label{conmod.2}
    \end{align}
\end{definition}
    Let $R$ be a Lie conformal superalgebra. A conformal $R$-module $M=M_\ep+M_\op$ is \textit{finite} if it is finitely generated as a \cp-module. We say that $M$ has \textit{rank} $(m+n)$, if, as \cp-modules, $M_\ep$ has rank $m$ and $M_\op$ has rank $n$. $M$ is called to be \textit{irreducible}, if $M$ has no nontrivial proper submodule.

    Obviously, for a fixed complex number $\eta$, the \cp-module $\bC c_\eta$ with $\partial c_\eta=\eta c_\eta$, $R_\lambda c_\eta=0$ is a conformal $R$-module. We refer to $\bC c_\eta$ as the \textit{even} (resp., \textit{odd}) \textit{one-dimensional trivial module} if $|c_\eta|=0$ (resp., $|c_\eta|=1$), denoted by $T_\eta^{ev}$ (resp., $T_\eta^{odd}$).

    \begin{proposition}\label{vir.mod}\textup{(\cite{CK})}
For the Virasoro Lie conformal algebra $\mathrm{Vir}$, all the free nontrivial $\mathrm{Vir}$-modules of rank 1 over $\bC[\partial]$ are the following ones $(\Delta,a\in\bC)$\textup{:}
\begin{align*}
    M_{\Delta,a}=\bC[\partial]v,\ L_\lambda v=(\partial+a+\Delta\lambda)v.
\end{align*}
The module $M_{\Delta,a}$ is irreducible if and only if $\Delta\neq 0$. The module $M_{0,a}$ contains a unique nontrivial submodule $(\partial+a)M_{0,\alpha}$ isomorphic to $M_{1,a}$. Moreover, the modules $M_{\Delta,a}$ with $\Delta\neq 0$ exhaust all finite irreducible nontrivial $\mathrm{Vir}$-modules.
\end{proposition}

\begin{proposition}\label{hv.mod}\textup{(\cite{SY})}
    For the Hisenberg--Virasoro Lie conformal algebra $\hv$, all the free nontrivial $\hv$-modules of rank 1 over $\bC[\partial]$ are the following ones $(\Delta,a,b\in\bC)$\textup{:}
\begin{align*}
    M_{\Delta,a,b}=\bC[\partial]v,\ L_\lambda v=(\partial+\Delta\lambda+a)v,\ H_\lambda v=bv.
\end{align*}
\end{proposition}

\section{Lie conformal superalgebras $\hvsi$ and $\hvsii$}\label{sec:Lie conformal salg}
\hspace{1.5em}In this section, we construct two  Lie conformal superalgebras related to the Heisenberg--Virasoro Lie conformal algebra. For a Lie conformal superalgebra $R$ of rank 3, we assume that $R_\ep=\mathcal{HV}=\cp L\oplus\cp H$, then $R_\op$ is of rank 1. We have the following results:
\begin{proposition}\label{classification}
    Let $R=\cp L\oplus\cp H\oplus\cp Y$ be a Lie conformal superalgebra that is a free \bZz-graded \cp-module, where $R_\ep=\cp L\oplus\cp H$  is the Heisenberg--Virasoro Lie conformal algebra and $R_\op=\cp Y$. Then $R$ is isomorphic to one of the following Lie conformal superalgebras\textup{:}

    \subno{1} $R_1$\textup{:} $\mathcal{HV}$, $[L_\lambda Y]=[H_\lambda Y]=[Y_\lambda Y]=0$,

    \subno{2} $R_2$\textup{:} $\mathcal{HV}$, $[L_\lambda Y]=(\partial+\lambda)Y$, $[H_\lambda Y]=0$, $[Y_\lambda Y]=\alpha H$, $\forall\ \alpha\in\bCx$,

    \subno{3} $R_3$\textup{:} $\mathcal{HV}$, $[L_\lambda Y]=(\partial+\beta\lambda+\gamma)Y$, $[H_\lambda Y]=\tau Y$, $[Y_\lambda Y]=0$, $\forall\ \beta,\gamma,\tau\in\bC$,

\noindent where $\mathcal{HV}$ means $L$ and $H$ satisfy the relations \eqref{HV}.
\end{proposition}
\begin{proof}
    Firstly we know that $\{L,H\}$ satisfy the relations \eqref{HV}, since $R_\ep$ is just an ordinary Lie conformal algebra for any Lie conformal superalgebra. Then we set $[L_\lambda Y]=a(\partial,\lambda)Y$, $[H_\lambda Y]=b(\partial,\lambda)Y$ and $[Y_\lambda Y]=d(\partial,\lambda)L+e(\partial,\lambda)H$, where $a(\partial,\lambda)=\sum_{i=0}^ma_i(\lambda)\partial^i$, $b(\partial,\lambda)=\sum_{i=0}^nb_i(\lambda)\partial^i$, $d(\partial,\lambda)=\sum_{i=0}^kd_i(\lambda)\partial^i$,
    $e(\partial,\lambda)=\sum_{i=0}^le_i(\lambda)\partial^i$ and they are all in $\bC[\partial,\lambda]$. Due to the skew-symmetry in \eqref{skew-symm}, we have $[Y_\lambda Y]=[Y_{-\partial-\lambda} Y]$. It leads to $d(\partial,\lambda)=d(\partial,-\partial-\lambda)$ and $e(\partial,\lambda)=e(\partial,-\partial-\lambda)$. Considering the degree in $\partial$ on the both sides of the two equalities respectively, we get $d(\partial,\lambda)=d(\partial)$ and $e(\partial,\lambda)=e(\partial)$. It follows that
    \begin{align}\label{YY}
        [Y_\lambda Y]=d(\partial)L+e(\partial)H,
    \end{align}
   where $d(\partial)=\sum_{i=0}^kd_i\partial^i$, $e(\partial)=\sum_{i=0}^le_i\partial^i$ and $d_i, e_i\in\bC$.

     From the Jacobi identity $[L_\lambda[L_\mu Y]]=[[L_\lambda L]_{\lambda+\mu}Y]+[L_\mu[L_\lambda Y]]$,  we have
    \begin{align*}
        [L_\lambda\big(a(\partial,\mu)Y\big)]=[\big((\partial+2\lambda)L\big)_{\lambda+\mu}Y]+[L_\mu\big(a(\partial,\lambda)Y\big)].
    \end{align*}
    By \eqref{conformal-sesqui} and comparing the coefficients of $Y$, we obtain
    \begin{align}\label{LLy}
        a(\partial+\lambda,\mu)a(\partial,\lambda)=(\lambda-\mu)a(\partial,\lambda+\mu)+a(\partial+\mu,\lambda)a(\partial,\mu).
    \end{align}
    Comparing the coefficients of $\partial^{2m-1}$ on the both sides of \eqref{LLy}, we conclude that $m(\lambda-\mu)a_m(\lambda)a_m(\mu)=0$ if $m>1$. It implies that $a_m(\lambda)=0$ for $m>1$. So $a(\partial,\lambda)=a_0(\lambda)+a_1(\lambda)\partial$. Substituting it into \eqref{LLy} and comparing the coefficients of $\partial$ and the constant terms respectively , we get
    \begin{align}
        a_1(\lambda)a_1(\mu)&=a_1(\lambda+\mu),\label{LLy.1}\\
        \lambda a_0(\lambda)a_1(\mu)-\mu a_1(\lambda)&a_0(\mu)=(\lambda-\mu)a_0(\lambda+\mu).\label{LLy.2}
    \end{align}
    By \eqref{LLy.1}, we have $a_1(\lambda)=0$ or 1. If $a_1(\lambda)=0$, then \eqref{LLy.2} becomes $(\lambda-\mu)a_0(\lambda+\mu)=0$, thus $a_0(\lambda)=0$, namely $[L_\lambda Y]=0$. If $a_1(\lambda)=1$, \eqref{LLy.2} turns into $\lambda a_0(\lambda)-\mu a_0(\mu)=(\lambda-\mu)a_0(\lambda+\mu)$. The degree of $a_0(\lambda)$ cannot be greater than 1 because the left hand of the equality does not contain the term of $\lambda\mu$. Therefore we can assume that $a_0(\lambda)=\beta\lambda+C_1$, where $\beta$ and $C_1$ are constants in \bC, namely $[L_\lambda Y]=(\partial+\beta\lambda+C_1)Y$.

    Using the Jacobi identity $[H_\lambda[H_\mu Y]]=[[H_\lambda H]_{\lambda+\mu}Y]+[H_\mu[H_\lambda Y]]$ and comparing the coefficients of $Y$, we get
    \begin{align*}
        b(\partial+\lambda,\mu)b(\partial,\lambda)=b(\partial+\mu,\lambda)b(\partial,\mu).
    \end{align*}
    Comparing the coefficients of $\partial^{2n-1}$ on the both sides, we obtain $n(\lambda-\mu)b_n(\lambda)b_n(\mu)=0$ if $n\ge1$, which means $b_n(\lambda)=0$ for $n\ge 1$. So $b(\partial,\lambda)=b_0(\lambda)$. Then $[H_\lambda Y]=b_0(\lambda)Y$.

    We need to discuss the following three cases:

    \afCase{1} $a_1(\lambda)=0$, namely $[L_\lambda Y]=0$.

    From the Jacobi identity  $[L_\lambda[Y_\mu Y]]=[[L_\lambda Y]_{\lambda+\mu}Y]+[Y_\mu[L_\lambda Y]]$ and \eqref{YY}, we obtain
    \begin{align*}
        (\partial+2\lambda)d(\partial+\lambda)L+(\partial+\lambda)e(\partial+\lambda)H=0.
    \end{align*}
    It follows that $d(\partial)=e(\partial)=0$, namely $[Y_\lambda Y]=0$. Now we consider the Jacobi identity  $[L_\lambda[H_\mu Y]]=[[L_\lambda H]_{\lambda+\mu}Y]+[H_\mu[L_\lambda Y]]$, this is equivalent to
    \begin{align*}
        \mu b_0(\lambda+\mu)Y=0.
    \end{align*}
    It forces $b_0(\lambda)=0$. Therefore $R$ is isomorphic to $R_1$.

    \afCase{2} $a_1(\lambda)=1$ and $b_0(\lambda)=0$, namely $[L_\lambda Y]=(\partial+\beta\lambda+C_1)Y$ and $[H_\lambda Y]=0$.

    Using the Jacobi identity  $[H_\lambda[Y_\mu Y]]=[[H_\lambda Y]_{\lambda+\mu}Y]+[Y_\mu[H_\lambda Y]]$, we get $\lambda d(\partial+\lambda)H=0$. It leads to $d(\partial)=0$, then $[Y_\lambda Y]=e(\partial)H$ by \eqref{YY}. Considering the Jacobi identity  $[L_\lambda[Y_\mu Y]]=[[L_\lambda Y]_{\lambda+\mu}Y]+[Y_\mu[L_\lambda Y]]$ and comparing the coefficients of $H$, we obtain
    \begin{align}\label{Lyy.2}
        (\partial+\lambda)e(\partial+\lambda)=(\partial+2\beta\lambda-\lambda+2C_1)e(\partial).
    \end{align}
    Equating the terms of the highest degree in $\lambda$ on the both sides of \eqref{Lyy.2}, we get $e(\partial)=C_2$, where $C_2$ is a constant.

    \subno{1} If $C_2\neq 0$, substituting $e(\partial)=C_2$ into \eqref{Lyy.2}, we get $\beta=1$ and $C_1=0$, namely $[L_\lambda Y]=(\partial+\lambda)Y$. Meanwhile we have $[H_\lambda Y]=0$ and $[Y_\lambda Y]=C_2 Y$. Then $R$ is isomorphic to $R_2$.

    \subno{2} If $C_2=0$, namely $[Y_\lambda Y]=0$. In addition, we have $[H_\lambda Y]=0$. Therefore it is a special case of $R_3$ if we take $\tau=0$.

    \afCase{3} $a_1(\lambda)=1$ and $b_0(\lambda)\neq 0$, namely $[L_\lambda Y]=(\partial+\beta\lambda+C_1)Y$ and $[H_\lambda Y]=b_0(\lambda)Y$.

    By the Jacobi identity  $[H_\lambda[Y_\mu Y]]=[[H_\lambda Y]_{\lambda+\mu}Y]+[Y_\mu[H_\lambda Y]]$ and comparing the coefficients of $L$ and $H$ respectively, we can get
    \begin{align}
        2b_0(\lambda)&d(\partial)=0,\label{Hyy.1}\\
        \lambda d(\partial+\lambda)&=2b_0(\lambda)e(\partial).\label{Hyy.2}
    \end{align}
    Because $b_0(\lambda)\neq 0$, \eqref{Hyy.1} forces $d(\partial)=0$. Plugging $d(\partial)=0$ into \eqref{Hyy.2}, we have $e(\partial)=0$. Therefore $[Y_\lambda Y]=0$ by \eqref{YY}. Due to the Jacobi identity  $[L_\lambda[H_\mu Y]]=[[L_\lambda H]_{\lambda+\mu}Y]+[H_\mu[L_\lambda Y]]$, we obtain
    \begin{align*}
        \mu b_0(\lambda+\mu)=\mu b_0(\mu).
    \end{align*}
    Thus $b_0(\lambda)=C_3$, where $C_3$ is a nonzero constant, namely $[H_\lambda Y]=C_3 Y$.

    Combining \textbf{Case 2}(2) and \textbf{Case 3}, we obtain that $R$ is isomorphic to $R_3$.
\end{proof}

The Lie conformal superalgebra $R_1$ in Proposition~\ref{classification} is the trivial case. We are more interested in the other two cases. So we construct two Lie conformal superalgebras related to the Heisenberg--Virasoro Lie conformal algebra in Definition~\ref{hsvi.def} and Definition~\ref{hsvii.def}, respectively. 

\begin{remark} 
Up to isomorphism, we may take $\alpha=2$ in $\hvsi$. At this time, the Lie conformal superalgebra $\mathcal{HVS}(2)$ is called the \textit{Heisenberg--Virasoro Lie conformal superalgebra} in \cite{CDH}, where $\hvsi$ was constructed by using the conformal module of Heisenberg--Virasoro Lie conformal algebra. For convenience, we abbreviate $\hvs(2)$ by $\hvs$ in the following. Interestingly, one Lie conformal superalgebra which is isomorphic to $\hvs$ was obtained as a lower truncated subalgebra of the twisted version of a class of $\bZ$-graded Lie conformal superalgebras in \cite{X22}.
\end{remark}

\section{Conformal derivations}\label{sect:conformal derivation}
\hspace{1.5em}In this section, we introduce the definition of conformal derivations and determine the conformal derivations of $\hvs$ and $\hvsii$.
\begin{definition}
    Let $R$ be a Lie conformal superalgebra. A conformal linear map $d_\lambda:R\rightarrow R$ is called a \textit{conformal derivation} if for any $a,b\in R$, it holds that
    \begin{align}\label{derivation}
        d_\lambda([a_\mu b])=[(d_\lambda a)_{\lambda+\mu}b]+(-1)^{|a||d|}[a_\mu(d_\lambda b)].
    \end{align}
\end{definition}

The space of all conformal derivations of $R$ are denoted by $\CDer(R)$. For any $a\in R$, one can define a linear map $(\ad a)_\lambda: R\rightarrow R$ by $(\ad a)_\lambda b=[a_\lambda b]$ for all $b\in R$. One can check easily that $(\ad a)_\lambda$ is a conformal derivation of $R$. Any conformal derivation of this form is called an \textit{inner derivation}. Denote by $\CInn(R)$ the space of all inner derivations of $R$.
\begin{theorem}\label{hvs.d}
    For the Heisenberg--Virasoro Lie conformal superalgebra \hvs, we have $$\CDer(\hvs)=\CInn(\hvs)\oplus D_\ep,$$ where $D_\ep=\{\dl\in\CDer(\hvs)_\ep\ |\ \dl L=C_1H,\ \dl H=\dl G=0,\ \forall\ C_1\in\bCx\}$.
\end{theorem}
\begin{proof}
    According to the degree of the conformal derivation $d$, two different cases should be discussed.

    \afCase{1} $|d|=0$.

    We can suppose that $\dl L=f_1(\partial,\lambda)L+f_2(\partial,\lambda)H$, where $f_i(\partial,\lambda) \in\bC[\partial,\lambda]\ (i=1,2)$. If $\dl\in\CDer(\hvs)$, then it has to satisfy $\dl[L_\mu L]=[(\dl L)_{\lambda+\mu} L]+[L_\mu(d_\lambda L)]$. By \eqref{conformal-sesqui}, the left hand of the equality becomes
    \begin{align*}
        \dl[L_\mu L]=\dl \big((\partial+2\mu)L\big)=(\partial+\lambda+2\mu)f_1(\partial,\lambda)L+(\partial+\lambda+2\mu)f_2(\partial,\lambda)H,
    \end{align*}
    and the right hand turns into
    \begin{align*}
        [(\dl L)_{\lambda+\mu} L]+[L_\mu(d_\lambda L)]&=[\big(f_1(\partial,\lambda)L+f_2(\partial,\lambda)H\big)_{\lambda+\mu}L]+[L_\mu\big(f_1(\partial,\lambda)L+f_2(\partial,\lambda)H\big)]\\
        &=(\partial+2\lambda+2\mu)f_1(-\lambda-\mu,\lambda)L+(\lambda+\mu)f_2(-\lambda-\mu,\lambda)H\\
        &\phantom{=\,}+(\partial+2\mu)f_1(\partial+\mu,\lambda)L+(\partial+\mu)f_2(\partial+\mu,\lambda)H.
    \end{align*}
    Then, comparing the coefficients of $L$ and $H$ respectively, we obtain
    \begin{align}
        (\partial+2\lambda+2\mu)f_1(-\lambda-\mu,\lambda)+(\partial+2\mu)f_1(\partial+\mu,\lambda)=(\partial+\lambda+2\mu)f_1(\partial,\lambda),\label{dLL.L}\\
        (\lambda+\mu)f_2(-\lambda-\mu,\lambda)+(\partial+\mu)f_2(\partial+\mu,\lambda)=(\partial+\lambda+2\mu)f_2(\partial,\lambda).\label{dLL.H}
    \end{align}
    Let $f_1(\partial,\lambda)=\sum_{i=0}^nf_{1,i}(\lambda)\partial^i$. Comparing the coefficients of $\partial^n$ on the both sides of \eqref{dLL.L}, we get $(\lambda-n\mu)f_{1,n}(\lambda)=0$ if $n>1$, which implies $f_{1,n}(\lambda)=0$ for $n>1$. Therefore $f_1(\partial,\lambda)=f_{1,0}(\lambda)+f_{1,1}(\lambda)\partial$. Similarly, we can deduce that $f_2(\partial,\lambda)=f_{2,0}(\lambda)$ by \eqref{dLL.H}. Replacing $\dl$ by $\dl-f_{1,1}(-\partial)(\ad L)_\lambda$, we have
    \begin{align*}
        \dl L=\big(f_{1,0}(\lambda)-2\lambda f_{1,1}(\lambda)\big)L+f_{2,0}(\lambda)H.
    \end{align*}
    In other words, there exist $f_1(\lambda), f_2(\lambda)\in\bC[\lambda]$, such that $\dl L=f_1(\lambda)L+f_2(\lambda)H$. Applying $\dl$ to $[L_\mu L]=(\partial+2\mu)L$, we get $(\partial+\lambda+2\mu)f_1(\lambda)L=0$ at this time, which shows that $f_1(\lambda)=0$, namely $\dl L=f_2(\lambda)H$. Assume $f_2(\lambda)=\lambda f(\lambda)+C_1$, where $f(\lambda)\in\bC[\lambda]$, $C_1\in\bC$. Then, replacing $\dl$ by $\dl-f(-\partial)(\ad H)_\lambda$, we have $\dl L=C_1H$.

    Now we can suppose that $\dl H=g_1(\partial,\lambda)L+g_2(\partial,\lambda)H$ and $\dl G=h(\partial,\lambda)G$, where $g_i(\partial,\lambda)(i=1,2)$, $h(\partial,\lambda)\in\bC[\partial,\lambda]$. Applying $\dl$ to $[L_\mu H]=(\partial+\mu)H$, we get
    \begin{align}
        (\partial+2\mu)g_1(\partial+\mu,\lambda)=(\partial+\lambda+\mu)g_1(\partial,\lambda),\label{dLH.L}\\
        (\partial+\mu)g_2(\partial+\mu,\lambda)=(\partial+\lambda+\mu)g_2(\partial,\lambda).\label{dLH.H}
    \end{align}
    Letting $\mu=0$ in \eqref{dLH.L} and \eqref{dLH.H} respectively, we obtain that $g_1(\partial,\lambda)=g_2(\partial,\lambda)=0$, namely $\dl H=0$. Applying $\dl$ to $[L_\mu G]=(\partial+\mu)G$, we get $(\partial+\mu)h(\partial+\mu,\lambda)=(\partial+\lambda+\mu)h(\partial,\lambda)$. Taking $\mu=0$, we can deduce that $h(\partial,\lambda)=0$ and $d_\lambda G=0$ consequently. Actually, if $C_1\neq 0$, the new $\dl$ is a conformal derivation but not an inner derivation.

    \afCase{2} $|d|=1$.

    We can suppose that $\dl L=p(\partial,\lambda)G$, where $p(\partial,\lambda)\in\bC[\partial,\lambda]$. Considering $d_\lambda[L_\mu L]=\dl\big((\partial+2\mu)L\big)$ and comparing the coefficients of $L$, we have
    \begin{align*}
        (\lambda+\mu)p(-\lambda-\mu,\lambda)+(\partial+\mu)p(\partial+\mu,\lambda)=(\partial+\lambda+2\mu)p(\partial,\lambda).
    \end{align*}
    This is similar to \eqref{dLL.H} and it gives that $p(\partial,\lambda)=p_0(\lambda)\in\bC[\lambda]$. Assume $p_0(\lambda)=\lambda p(\lambda)+C_2$, where $p(\lambda)\in\bC[\lambda]$, $C_2\in\bC$. Replacing $\dl$ by $\dl-p(-\partial)(\ad G)_\lambda$, we have $\dl L=C_2G$.

    Now we can suppose that $\dl H=q(\partial,\lambda)G$ and $d_\lambda G=r(\partial,\lambda)L+s(\partial,\lambda)H$, where $q(\partial,\lambda)$, $r(\partial,\lambda)$ and $s(\partial,\lambda)$ are all in $\bC[\partial,\lambda]$. Applying $\dl$ to $[L_\mu H]=(\partial+\mu)H$, we get
    \begin{align}\label{dLH.1}
         (\partial+\mu)q(\partial+\mu,\lambda)=(\partial+\lambda+\mu)q(\partial,\lambda).
    \end{align}
    This is similar to \eqref{dLH.H} and it leads to $\dl H=0$. Applying $\dl$ to $[H_\mu G]=0$, we get $\mu r(\partial+\mu,\lambda)H=0$. It follows that $r(\partial,\lambda)=0$, namely $\dl G=s(\partial,\lambda)H$. Applying $\dl$ to $[L_\mu G]=(\partial+\mu)G$, we have
    \begin{align}\label{dLG}
        2C_2+(\partial+\mu)s(\partial+\mu,\lambda)=(\partial+\lambda+\mu)s(\partial,\lambda).
    \end{align}
    Letting $\mu=0$ in \eqref{dLG}, it implies that $2C_2=\lambda s(\partial,\lambda)$. Therefore, if and only if $C_2=s(\partial,\lambda)=0$, the equality holds. In other words, the conformal derivations of degree one must be inner derivations.
\end{proof}
\begin{theorem}\label{hvsii.d}
    For the Lie conformal superalgebra \hvsii, we have
    \begin{align*}
        \CDer\big(\hvsii\big)=\CInn\big(\hvsii\big)\oplus\delta_{\tau,0}D_\ep\oplus \delta_{\beta,1}\delta_{\tau,0}D_\op,
    \end{align*}
    where $D_\ep=\{\dl\in\CDer\big(\hvsii\big)_\ep\ |\ \dl L=C_3H,\ \dl H=\dl E=0,\ \forall\ C_3\in\bCx\}$ and $D_\op=\{\dl\in\CDer\big(\hvsii\big)_\op\ |\ \dl L=C_4E,\ \dl H=\dl E=0,\ \forall\  C_4\in\bCx\}$.
\end{theorem}
\begin{proof}
    Two cases need to be discussed based on the degree of the conformal derivation $d$.

    \afCase{1} $|d|=0$.

    We suppose that $\dl L=f_1(\partial,\lambda)L+f_2(\partial,\lambda)H$, where $f_i(\partial,\lambda) \in\bC[\partial,\lambda]\ (i=1,2)$. Since $[L_\lambda L]=(\partial+2\lambda)L$, which is similar to the case of \hvs, we can deduce that $f_1(\partial,\lambda)=f_{1,0}(\lambda)+f_{1,1}(\lambda)\partial$ and $f_2(\partial,\lambda)=f_{2,0}(\lambda)$. Assume $f_{2,0}(\lambda)=\lambda f(\lambda)+C_3$, where $f(\lambda)\in\bC[\lambda]$, $C_3\in\bC$. Similar to the discussion of Case~1 in Lemma~\ref{hvs.d}, $\dl$ can be replaced by $\dl-f_{1,1}(-\partial)(\ad L)_\lambda-f(-\partial)(\ad H)_\lambda$ and we have $\dl L=C_3H$ as a consequence. Then we suppose that $\dl H=g_1(\partial,\lambda)L+g_2(\partial,\lambda)H$, where $g_i(\partial,\lambda) \in\bC[\partial,\lambda]\ (i=1,2)$. Applying $\dl$ to $[L_\mu H]=(\partial+\mu)H$, similarly, we can get \eqref{dLH.L} and \eqref{dLH.H}, which leads to $\dl H=0$ consequently.

    Now we can assume that $\dl E=h(\partial,\lambda)E$, where $h(\partial,\lambda)\in\bC[\partial,\lambda]$. Applying $\dl$ to $[H_\mu E]=\tau E$, we obtain
    \begin{align}\label{dHE}
        \tau h(\partial+\mu,\lambda)=\tau h(\partial,\lambda).
    \end{align}

    If $\tau\neq 0$, because the right hand of \eqref{dHE} does not contain $\mu$, we have $h(\partial,\lambda)=h_0(\lambda)\in\bC[\lambda]$ and $\dl E=h_0(\lambda)E$. Applying $\dl$ to $[L_\mu E]=(\partial+\beta\mu+\gamma)E$, we get $\tau C_3=\lambda h_0(\lambda)$. Since $\tau\neq 0$, the equality holds if and only if $C_3=h_0(\lambda)=0$. At this time $\dl L=\dl E=0$, which shows that the conformal derivations must be inner derivations.

    If $\tau=0$, applying $\dl$ to $[L_\mu E]=(\partial+\beta\mu+\gamma)E$ again, we get
    \begin{align}\label{dLE.0}
        (\partial+\beta\mu+\gamma)h(\partial+\mu,\lambda)=(\partial+\lambda+\beta\mu+\gamma)h(\partial,\lambda).
    \end{align}
    Taking $\mu=0$ in \eqref{dLE.0}, we obtain $h(\partial,\lambda)=0$, namely $\dl E=0$. If $C_3\neq 0$, we find the conformal derivation which is not an inner derivation.

    \afCase{2} $|d|=1$.

    We can suppose that $\dl L=p(\partial,\lambda)E$, where $p(\partial,\lambda)\in\bC[\partial,\lambda]$. Applying $\dl$ to $[L_\mu L]=(\partial+2\mu)L$, we get
    \begin{align}\label{dLL.1.hvsii}
        &(\partial+\lambda+2\mu)p(\partial,\lambda)-(\partial+\beta\mu+\gamma)p(\partial+\mu,\lambda)\notag\\
        &=\big((\beta-1)\partial+\beta(\lambda+\mu)-\gamma\big)p(-\lambda-\mu,\lambda).
    \end{align}
     Let $p(\partial,\lambda)=\sum_{i=0}^np_i(\lambda)\partial^i$. Comparing the coefficients of $\partial^n$ on the both sides of \eqref{dLL.1.hvsii}, we obtain that $\big((\beta+n-2)\mu-\lambda+\gamma\big)p_n(\lambda)=0$ if $n>1$, namely $p_n(\lambda)=0$ for $n>1$. Therefore $p(\partial,\lambda)=p_0(\lambda)+p_1(\lambda)\partial$ is linear in $\partial$. Plugging this into \eqref{dLL.1.hvsii}, we can get
    \begin{align*}
        \big((1-\beta)p_0(\lambda)+(\beta\lambda-\gamma)p_1(\lambda)\big)\big((\lambda+2\mu)\partial+1\big)=0.
    \end{align*}
    This further demonstrates that
    \begin{align}\label{dLL}
        (\beta-1)p_0(\lambda)=(\beta\lambda-\gamma)p_1(\lambda).
    \end{align}

    We need to discuss the following two cases according to the value of $\beta$.

    \subno{1} If $\beta\neq 1$, then $p_0(\lambda)=\frac{\beta\lambda-\gamma}{\beta-1}p_1(\lambda)$ by \eqref{dLL}, namely $p(\partial,\lambda)=(\partial+\frac{\beta}{\beta-1}\lambda-\frac{\gamma}{\beta-1})p_1(\lambda)$. Replacing $\dl$ by $\dl-\frac{p_1(-\partial)}{\beta-1}(\ad E)_\lambda$, we have $\dl L=0$. We can suppose that $\dl H=q(\partial,\lambda)E$ and $d_\lambda E=r(\partial,\lambda)L+s(\partial,\lambda)H$, where $q(\partial,\lambda)$, $r(\partial,\lambda)$ and $s(\partial,\lambda)$ are all in $\bC[\partial,\lambda]$. Applying $\dl$ to $[L_\mu H]=(\partial+\mu)H$, we have
    \begin{align*}
        (\partial+\beta\mu+\gamma)q(\partial+\mu,\lambda)=(\partial+\lambda+\mu)q(\partial,\lambda).
    \end{align*}
    Taking $\mu=0$ in this equality, we can get $(\lambda-\gamma)q(\partial,\lambda)=0$, which means $q(\partial,\lambda)=0$ and $\dl H=0$. Applying $\dl$ to $[H_\mu E]=\tau E$, we have $\mu r(\partial+\mu,\lambda)H=\tau r(\partial,\lambda)L+\tau s(\partial,\lambda)H$. Comparing the coefficients of $L$ and $H$ respectively, we obtain that
    \begin{align}
        \tau r(\partial,\lambda)&=0,\label{dHE.L}\\
        \mu r(\partial+\mu,\lambda)&=\tau s(\partial,\lambda).\label{dHE.H}
    \end{align}
    If $\tau\neq 0$, by \eqref{dHE.L} and \eqref{dHE.H}, we have $r(\partial,\lambda)=s(\partial,\lambda)=0$, namely $\dl E=0$. If $\tau=0$, then $r(\partial,\lambda)=0$ by \eqref{dHE.H}. It is equivalent to saying that $\dl E=s(\partial,\lambda)H$. Now applying $\dl$ to $[L_\mu E]=(\partial+\beta\mu+\gamma)E$, we get
    \begin{align}\label{dLE}
        (\partial+\mu)s(\partial+\mu,\lambda)=(\partial+\lambda+\beta\mu+\gamma)s(\partial,\lambda).
    \end{align}
    Taking $\mu=0$ in \eqref{dLE}, then we obtain $(\lambda+\gamma)s(\partial,\lambda)=0$, which implies that $s(\partial,\lambda)=0$. Therefore we have $\dl E=0$ whatever $\tau$ takes.

    \subno{2} If $\beta=1$, then $p_1(\lambda)=0$ by \eqref{dLL}, namely $p(\partial,\lambda)=p_0(\lambda)$. Assume that $p_0(\lambda)=(\lambda-\gamma)p(\lambda)+p_0(\gamma)$, where $p(\lambda)\in\bC[\lambda]$. Two cases are discussed as follows.

     \subno{i} If $(\lambda-\gamma)\mid p_0(\lambda)$, we have $p_0(\gamma)=0$, namely $p_0(\lambda)=(\lambda-\gamma)p(\lambda)$. Replacing $\dl$ by $\dl-p(-\partial)(\ad E)_\lambda$, we get $\dl L=0$. Similarly, one can easily obtain that $\dl H=\dl E=0$.

    \subno{ii} If $(\lambda-\gamma)\nmid p_0(\lambda)$, replacing $\dl$ by $\dl-p(-\partial)(\ad E)_\lambda$, we can deduce that $\dl L=p_0(\gamma)E$ and $p_0(\gamma)\neq 0$. We still suppose that $\dl H=q(\partial,\lambda)E$ and $d_\lambda E=r(\partial,\lambda)L+s(\partial,\lambda)H$, where $q(\partial,\lambda)$, $r(\partial,\lambda)$ and $s(\partial,\lambda)$ are all in $\bC[\partial,\lambda]$. Now applying $\dl$ to $[H_\mu H]=0$, we can obtain that $\tau q(-\lambda-\mu,\lambda)=\tau q(\partial+\mu,\lambda)$ which means the value of $q(\partial,\lambda)$ is independent of $\partial$. So we can assume $\dl H=q(\lambda)E$, where $q(\lambda)\in\bC[\lambda]$. Applying $\dl$ to $[L_\mu H]=(\partial+\mu)H$, we get
    \begin{align}\label{dLH.d1}
        \tau p_0(\gamma)=(\gamma-\lambda)q(\lambda).
    \end{align}
    If $\tau\neq 0$, then it gives that $q(\lambda)\neq 0$ since  $p_0(\gamma)$ is nonzero. However, the degree in $\lambda$ on the right hand of \eqref{dLH.d1} must be greater than the left, obtaining a contradiction. Therefore we only need to consider the case of $\tau=0$. If $\tau=0$, applying $\dl$ to $[L_\mu H]=(\partial+\mu)H$, we obtain
    \begin{align}\label{dLE.d1.t-1}
        (\partial+\mu+\gamma)q(\partial+\mu,\lambda)=(\partial+\lambda+\mu)q(\partial,\lambda).
    \end{align}
    Letting $\mu=0$ in \eqref{dLE.d1.t-1}, we can get $(\lambda-\gamma)q(\partial,\lambda)=0$, which implies that $q(\partial,\lambda)=0$ and $\dl H=0$. Applying $\dl$ to $[H_\mu E]=0$, we have $\mu r(\partial+\mu,\lambda)H=0$, namely $r(\partial,\lambda)=0$. It leads to $\dl E=s(\partial,\lambda)H$. Applying $\dl$ to $[L_\mu E]=(\partial+\mu+\gamma)E$, we get
    \begin{align}\label{dLE.d1.t0}
        (\partial+\mu)s(\partial+\mu,\lambda)=(\partial+\lambda+\mu+\gamma)s(\partial,\lambda).
    \end{align}
    Taking $\mu=0$ in \eqref{dLE.d1.t0}, we can conclude that $(\lambda+\gamma)s(\partial,\lambda)=0$ and $\dl E=0$. Therefore we find a conformal derivation which is not an inner derivation. Due to the arbitrariness of $p_0(\lambda)$ in $\bC[\lambda]$, we can take $p_0(\gamma)=C_4\in\bCx$.
\end{proof}

The following conclusions can be easily obtained by referring to the proof of Theorems~\ref{hvs.d} and \ref{hvsii.d}.
\begin{corollary}
    Let $R$ be a Lie conformal superalgebra.

    \subno{1} If $R$ is isomorphic to the Lie conformal superalgebra $R_4$ in Proposition~\ref{rank1.1}, then
    \begin{align*}
        \CDer(R)=\CInn(R)\oplus\delta_{\beta,1}D_\op,
    \end{align*}
    where $D_\op=\{\dl\in\CDer(R)_\op\ |\ \dl x=cy,\ \dl y=0,\ c\in\bCx\}$.

    \subno{2} If $R$ is isomorphic to the Lie conformal superalgebra $R_5$ in Proposition~\ref{rank1.1}, then
    \begin{align*}
        \CDer(R)=\CInn(R).
    \end{align*}
\end{corollary}

\section{Conformal biderivations}\label{sect:conformal biderivation}
\hspace{1.5em}In this section, we firstly introduce the definition of conformal biderivations and an important lemma, one can refer to \cite{H} for more details. Then we determine the conformal biderivations of $\hvs$ and $\hvsii$, respectively.
\begin{definition}
    Let $R$ be a Lie conformal superalgebra. A conformal bilinear map $\varphi_\lambda: R\times R\rightarrow R[\lambda]$ is called a \textit{conformal biderivation} if for any $a, b, c\in R$, it satisfies the following relations:
    \begin{align*}
        \varphi_\lambda(a,b)&=-(-1)^{|a||b|}\varphi_{-\partial-\lambda}(b,a),\\
        \varphi_\lambda(a,[b_\mu c])&=[\varphi_\lambda(a,b)_{\lambda+\mu}c]+(-1)^{|a||b|}[b_\mu\varphi_\lambda(a,c)].
    \end{align*}
    For any $a, b\in R$, if there exists a fixed complex number $\varepsilon$ such that $\varphi_\lambda^\varepsilon(a,b)=\varepsilon[a_\lambda b]$, then we call $\varphi_\lambda^\varepsilon$ an \textit{inner conformal biderivation} of $R$. The space of all conformal biderivations and inner conformal biderivations of $R$ are denoted by $\BDer(R)$ and $\BInn(R)$, respectively.
\end{definition}
\begin{lemma}\label{lemma:bider}\textup{(\cite{H})}
    Let $\varphi_\lambda$ be a conformal biderivation of a Lie conformal superalgebra $R$.

    \subno{1} For any $a,b,c,d\in R$,
    \begin{align}\label{bider}
    [\varphi_\lambda(a,b)_{\lambda+\nu}[c_\mu d]]=[[a_\lambda b]_{\lambda+\nu}\varphi_\mu(c,d)].
    \end{align}

    \subno{2} For any $a,b\in R$, if $[a_\lambda b]=0$, then $\varphi_\lambda(a,b)\in Z[R_\lambda R]$, where $Z[R_\lambda R]$ is the center of $[R_\lambda R]$.
\end{lemma}

\begin{lemma}\label{bider.hvs}
    For the Lie conformal superalgebra \hvs, its conformal biderivation $\varphi_\lambda$ has the following forms:
    \begin{align*}
        \varphi_\lambda(L,L)=C(\partial+2\lambda)L,&\ \varphi_\lambda(L,H)=C(\partial+\lambda)H,\ \varphi_\lambda(L,G)=C(\partial+\lambda)G,\\
        \varphi_\lambda(G,G)&=2CH,\ \varphi_\lambda(H,H)=\varphi_\lambda(H,G)=0,
    \end{align*}
    where $C\in\bC$.
\end{lemma}
\begin{proof}
    Let $\varphi_\lambda$ be a conformal biderivation of \hvs. Since the center of $[\hvs_\lambda\hvs]$ is trivial, by Lemma~\ref{lemma:bider}(2), we have $\varphi_\lambda(H,H)=\varphi_\lambda(H,G)=0$.

    We can assume that
    \begin{align*}
        \varphi_\lambda(L,L)=f_{L,L}(\partial,\lambda)L+g_{L,L}(\partial,\lambda)H+h_{L,L}(\partial,\lambda)G,
    \end{align*}
    where $f_{L,L}(\partial,\lambda),g_{L,L}(\partial,\lambda)$ and $h_{L,L}(\partial,\lambda)$ are all in $\bC[\partial,\lambda]$. Taking $a=b=c=d=L$ in \eqref{bider} and comparing the coefficients of $L, H, G$ respectively, we obtain
    \begin{align}
        (\partial+\lambda+\nu+2\mu)f_{L,L}(-\lambda-\nu,\lambda)&=(\lambda-\nu)f_{L,L}(\partial+\lambda+\nu,\mu),\label{LLLL.L}\\
        (\lambda+\nu)(\partial+\lambda+\nu+2\mu)g_{L,L}(-\lambda-\nu,\lambda)&=(\lambda-\nu)(\partial+\lambda+\nu)g_{L,L}(\partial+\lambda+\nu,\mu),\label{LLLL.H}\\
        (\lambda+\nu)(\partial+\lambda+\nu+2\mu)h_{L,L}(-\lambda-\nu,\lambda)&=(\lambda-\nu)(\partial+\lambda+\nu)h_{L,L}(\partial+\lambda+\nu,\mu).\label{LLLL.G}
    \end{align}

    Firstly, by comparing the degree in $\partial$ on the both sides of \eqref{LLLL.L}, we know that the degree in $\partial$ of $f_{L,L}(\partial,\lambda)$ cannot be greater than 1. Thus $f_{L,L}(\partial,\lambda)=f_0(\lambda)+f_1(\lambda)\partial$ is linear in $\partial$, where $f_0(\lambda),f_1(\lambda)\in\bC[\lambda]$. Substitute it into \eqref{LLLL.L}, we can deduce that
    \begin{align*}
        (\partial+\lambda+\nu+2\mu)\big(f_0(\lambda)+f_1(\lambda)(-\lambda-\nu)\big)=(\lambda-\nu)\big(f_0(\mu)+f_1(\mu)(\partial+\lambda+\nu)\big).
    \end{align*}
    Comparing its coefficients of $\partial$, we have
    \begin{align}\label{LLLL.L.2}
        f_0(\lambda)-(\lambda+\nu)f_1(\lambda)=(\lambda-\nu)f_1(\mu).
    \end{align}
    Putting $\lambda=\nu$ in \eqref{LLLL.L.2}, we get $f_0(\nu)=2\nu f_1(\nu)$. Plugging this into \eqref{LLLL.L.2}, we obtain $(\lambda-\nu)(f_1(\lambda)-f_1(\mu))=0$, which means $f_1(\lambda)=C$, where $C\in\bC$ is a constant. Therefore $f_{L,L}(\partial,\lambda)=C(\partial+2\lambda)$.

    Then, taking $\mu=0$ in \eqref{LLLL.H}, we get
    \begin{align}\label{LLLL.H.1}
        (\lambda+\nu)g_{L,L}(-\lambda-\nu,\lambda)=(\lambda-\nu)g_{L,L}(\partial+\lambda+\nu,0).
    \end{align}
    Comparing the degree in $\partial$ on the both sides of \eqref{LLLL.H.1}, we have $g_{L,L}(\partial,\lambda)=g(\lambda)\in\bC[\lambda]$. This together with \eqref{LLLL.H.1}, shows that $g_{L,L}(\partial,\lambda)=0$. Similarly, it follows from \eqref{LLLL.G} that $h_{L,L}(\partial,\lambda)=0$. Therefore $\varphi_\lambda(L,L)=C(\partial+2\lambda)L$.

    Now, we assume
    \begin{align*}
        \varphi_\lambda(L,H)=f_{L,H}(\partial,\lambda)L+g_{L,H}(\partial,\lambda)H+h_{L,H}(\partial,\lambda)G,
    \end{align*}
    where $f_{L,H}(\partial,\lambda),g_{L,H}(\partial,\lambda)$ and $h_{L,H}(\partial,\lambda)$ are all in $\bC[\partial,\lambda]$. Taking $a=b=c=L$ and $d=H$ in \eqref{bider} and comparing the coefficients of $L, G, H$ respectively, we get $f_{L,H}(\partial,\lambda)=h_{L,H}(\partial,\lambda)=0$ and $g_{L,H}(\partial+\lambda+\nu,\mu)=C(\partial+\lambda+\nu+\mu)$, which implies $g_{L,H}(\partial,\lambda)=C(\partial+\lambda)$. Thus $\varphi_\lambda(L,H)=C(\partial+\lambda)H$. Similarly, we can derive that $\varphi_\lambda(L,G)=C(\partial+\lambda)G$.

     Finally, we discuss the case of $\varphi_\lambda(G,G)$. Suppose that
     \begin{align*}
         \varphi_\lambda(G,G)=f_{G,G}(\partial,\lambda)L+g_{G,G}(\partial,\lambda)H+h_{G,G}(\partial,\lambda)G,
     \end{align*}
      where $f_{G,G}(\partial,\lambda),g_{G,G}(\partial,\lambda)$ and $h_{G,G}(\partial,\lambda)$ are all in $\bC[\partial,\lambda]$. Taking $a=b=L$ and $c=d=G$ in \eqref{bider} and comparing the coefficients of $L, G, H$ respectively, we obtain $f_{G,G}(\partial,\lambda)=h_{G,G}(\partial,\lambda)=0$ and $g_{G,G}(\partial+\lambda+\nu,\mu)=2C$. It follows that $g_{G,G}(\partial,\mu)=2C$. Therefore we have $g_{G,G}(\partial,\lambda)=2C$ and $\varphi_\lambda(G,G)=2CH$.
\end{proof}

\begin{lemma}\label{bider.hvsii}
     For the Lie conformal superalgebra \hvsii, its conformal biderivation $\varphi_\lambda$ has the following forms:
    \begin{align*}
        \varphi_\lambda(L,L)&=C(\partial+2\lambda)L+\delta_{\beta,2}\delta_{\gamma,0}\delta_{\tau,0}\bar{C}(\partial+2\lambda)E,\\
        \varphi_\lambda(L,H)&=C(\partial+\lambda)H,\ \ \varphi_\lambda(L,E)=C(\partial+\beta\lambda+\gamma)E,\\
        \varphi_\lambda(H,E)&=C\tau E,\ \ \varphi_\lambda(H,H)=\varphi_\lambda(E,E)=0,
    \end{align*}
    where $C,\bar{C}\in\bC$.
\end{lemma}
\begin{proof}
     Let $\varphi_\lambda$ be a conformal biderivation of \hvsii. $Z[\hvsii_\lambda\hvsii]$ is trivial, so we have $\varphi_\lambda(H,H)=\varphi_\lambda(E,E)=0$ by Lemma~\ref{lemma:bider}(2).

     Assume that
    \begin{align*}
        \varphi_\lambda(L,L)=f_{L,L}(\partial,\lambda)L+g_{L,L}(\partial,\lambda)H+h_{L,L}(\partial,\lambda)E,
    \end{align*}
    where $f_{L,L}(\partial,\lambda),g_{L,L}(\partial,\lambda)$ and $h_{L,L}(\partial,\lambda)$ are all in $\bC[\partial,\lambda]$. Taking $a=b=c=d=L$ in \eqref{bider} and comparing the coefficients of $L, H, E$ respectively, we obtain \eqref{LLLL.L}, \eqref{LLLL.H} and
    \begin{align}\label{LLLL.E}
        &\big((\beta-1)\partial+\beta(\lambda+\nu)-\gamma\big)(\partial+\lambda+\nu+2\mu)h_{L,L}(-\lambda-\nu,\lambda)\notag\\
        &\ =(\lambda-\nu)\big(\partial+\beta(\lambda+\nu)+\gamma\big)h_{L,L}(\partial+\lambda+\nu,\mu).
    \end{align}
     We can deduce that $f_{L,L}(\partial,\lambda)=C(\partial+2\lambda)$ and $g_{L,L}(\partial,\lambda)=0$ via the discussion about \eqref{LLLL.L} and \eqref{LLLL.H} in Lemma~\ref{bider.hvs}, where $C\in\bC$ is a constant.

    By comparing the degree in $\partial$ on the both sides of \eqref{LLLL.E}, we know that the degree in $\partial$ of $h_{L,L}(\partial,\lambda)$ cannot be greater than 1. Thus $h_{L,L}(\partial,\lambda)=h_0(\lambda)+h_1(\lambda)\partial$ is linear in $\partial$, where $h_0(\lambda),h_1(\lambda)\in\bC[\lambda]$. Substituting it into \eqref{LLLL.E}, we have
    \begin{align}\label{brt.1}
         &\big((\beta-1)\partial+\beta(\lambda+\nu)-\gamma\big)(\partial+\lambda+\nu+2\mu)\big(h_0(\lambda)-h_1(\lambda)(\lambda+\nu)\big)\notag\\
        &=\ (\lambda-\nu)\big(\partial+\beta(\lambda+\nu)+\gamma\big)\big(h_0(\mu)+h_1(\mu)(\partial+\lambda+\nu)\big).
    \end{align}
    Putting $\lambda=\nu$ in \eqref{brt.1}, we can obtain
    \begin{align*}
        \big((\beta-1)\partial+2\beta\lambda-\gamma\big)(\partial+2\lambda+2\mu)\big(h_0(\lambda)-2\lambda h_1(\lambda)\big)=0,
    \end{align*}
    which means that $h_0(\lambda)=2\lambda h_1(\lambda)$. This together with \eqref{brt.1}, shows that
    \begin{align}\label{brt.2}
        \big((\beta-1)\partial+\beta(\lambda+\nu)-\gamma\big)h_1(\lambda)=\big(\partial+\beta(\lambda+\nu)+\gamma\big)h_1(\mu).
    \end{align}
    Observing the equality \eqref{brt.2}, we can get $h_1(\lambda)$ is nonzero only if $\beta=2$ and $\gamma=0$. Thus we have to discuss the following two cases.

    \afCase{1} $\beta=2$ and $\gamma=0$, namely $[L_\lambda E]=(\partial+2\lambda)E$.

    We have $h_1(\lambda)=h_1(\mu)=\bar{C}$ by \eqref{brt.2}, where $\bar{C}\in\bC$ is a constant. It follows that $h_0(\lambda)=2\bar{C}\lambda$ and $h_{L,L}(\partial,\lambda)=\bar{C}(\partial+2\lambda)$. So we get $\varphi_\lambda(L,L)=(\partial+2\lambda)(CL+\bar{C}E)$.

    Now we suppose that
    \begin{align*}
         \varphi_\lambda(L,H)=f_{L,H}(\partial,\lambda)L+g_{L,H}(\partial,\lambda)H+h_{L,H}(\partial,\lambda)E,
    \end{align*}
    where $f_{L,H}(\partial,\lambda),g_{L,H}(\partial,\lambda)$ and $h_{L,H}(\partial,\lambda)$ are all in $\bC[\partial,\lambda]$. Taking $a=b=c=L$ and $d=H$ in \eqref{bider} and comparing the coefficients of $L, H, E$ respectively, we obtain $f_{L,H}(\partial,\lambda)=0$, $g_{L,H}(\partial,\lambda)=C(\partial+\lambda)$ and
    \begin{align}
         (\partial+2\lambda+2\nu)h_{L,H}(\partial+\lambda+\nu,\mu)=-\tau\bar{C}(\partial+\lambda+\nu+\mu).\label{LLLH.E}
    \end{align}
     Putting $\lambda+\nu=\mu$ in \eqref{LLLH.E}, we can derive that $h_{L,H}(\partial+\mu,\mu)=-\tau\bar{C}$, which implies $h_{L,H}(\partial,\lambda)=-\tau\bar{C}$. Plugging this into \eqref{LLLH.E}, we have $\tau\bar{C}(\lambda+\nu-\mu)=0$. This further demonstrates that $\tau\bar{C}=h_{L,H}(\partial,\lambda)=0$ and $\varphi_\lambda(L,H)=C(\partial+\lambda)H$. Note that if $\tau\neq 0$, then $\bar{C}=0$ and $\varphi_\lambda(L,L)=C(\partial+2\lambda)L$.

    Next we suppose that
    \begin{align*}
        \varphi_\lambda(L,E)=f_{L,E}(\partial,\lambda)L+g_{L,E}(\partial,\lambda)H+h_{L,E}(\partial,\lambda)E,
    \end{align*}
    where $f_{L,E}(\partial,\lambda),g_{L,E}(\partial,\lambda)$ and $h_{L,E}(\partial,\lambda)$ are all in $\bC[\partial,\lambda]$. Taking $a=b=c=L$ and $d=E$ in \eqref{bider} and comparing the coefficients of $L, H, E$ respectively, we get $f_{L,E}(\partial,\lambda)=g_{L,E}(\partial,\lambda)=0$ and $h_{L,E}(\partial+\lambda+\nu,\mu)=C(\partial+\lambda+\nu+2\mu)$. It follows that $h_{L,E}(\partial,\lambda)=C(\partial+2\lambda)$ and $\varphi_\lambda(L,E)=C(\partial+2\lambda)E$.

     Then we suppose that
    \begin{align*}
        \varphi_\lambda(H,E)=f_{H,E}(\partial,\lambda)L+g_{H,E}(\partial,\lambda)H+h_{H,E}(\partial,\lambda)E,
    \end{align*}
    where $f_{H,E}(\partial,\lambda),g_{H,E}(\partial,\lambda)$ and $h_{H,E}(\partial,\lambda)$ are all in $\bC[\partial,\lambda]$. Taking $a=b=L$, $c=H$ and $d=E$ in \eqref{bider} and comparing the coefficients of $L, H, E$ respectively, we obtain $f_{H,E}(\partial,\lambda)=g_{H,E}(\partial,\lambda)=0$ and $h_{H,E}(\partial+\lambda+\nu,\mu)=C\tau$, which means $h_{H,E}(\partial,\lambda)=C\tau$. Thus we have $\varphi_\lambda(H,E)=C\tau E$.

    \afCase{2} $\beta\neq 2$ or $\gamma\neq 0$.

     It follows from \eqref{brt.2} that $h_1(\lambda)=h_1(\mu)=0$, namely $h_0(\lambda)=h_1(\lambda)=0$. Therefore $h_{L,L}(\partial,\lambda)=0$ and we can deduce that $\varphi_\lambda(L,L)=C(\partial+2\lambda)L$. Similarly, we can get $\varphi_\lambda(L,H)=C(\partial+\lambda)H$.

    For the case of $\varphi_\lambda(L,E)$, we suppose that
    \begin{align*}
        \varphi_\lambda(L,E)=f_{L,E}(\partial,\lambda)L+g_{L,E}(\partial,\lambda)H+h_{L,E}(\partial,\lambda)E,
    \end{align*}
    where $f_{L,E}(\partial,\lambda),g_{L,E}(\partial,\lambda)$ and $h_{L,E}(\partial,\lambda)$ are all in $\bC[\partial,\lambda]$. Taking $a=b=c=L$ and $d=E$ in \eqref{bider} and comparing the coefficients of $L, H, E$ respectively, we obtain $f_{L,E}(\partial,\lambda)=g_{L,E}(\partial,\lambda)=0$ and $h_{L,E}(\partial+\lambda+\nu,\mu)=C(\partial+\lambda+\nu+\beta\mu+\gamma)$. It implies that $h_{L,E}(\partial,\lambda)=C(\partial+\beta\lambda+\gamma)$, namely $\varphi_\lambda(L,E)=C(\partial+\beta\lambda+\gamma)E$.

    Then suppose that
    \begin{align*}
        \varphi_\lambda(H,E)=f_{H,E}(\partial,\lambda)L+g_{H,E}(\partial,\lambda)H+h_{H,E}(\partial,\lambda)E,
    \end{align*}
    where $f_{H,E}(\partial,\lambda),g_{H,E}(\partial,\lambda)$ and $h_{H,E}(\partial,\lambda)$ are all in $\bC[\partial,\lambda]$. Taking $a=b=L$, $c=H$ and $d=E$ in \eqref{bider} and comparing the coefficient of $L, H, E$ respectively, we get $f_{H,E}(\partial,\lambda)=g_{H,E}(\partial,\lambda)=0$ and $h_{H,E}(\partial+\lambda+\nu,\mu)=C\tau$. It follows that $h_{H,E}(\partial,\lambda)=C\tau$. It is equivalent to saying that $h_{H,E}(\partial,\lambda)=C\tau$. Thus we have $\varphi_\lambda(H,E)=C\tau E$.
\end{proof}

By Lemmas~\ref{bider.hvs} and \ref{bider.hvsii}, we have the following theorem.
\begin{theorem}\label{thm:biderivation}
The following statements hold:

\subno{1} For the Heisenberg--Virasoro Lie conformal superalgebra \hvs, we have
    \begin{align*}
        \BDer(\hvs)=\BInn(\hvs).
    \end{align*}

\subno{2} For the Lie conformal superalgebra \hvsii, we have
    \begin{align*}
        \BDer\big(\hvsii\big)=\BInn\big(\hvsii\big)\oplus\delta_{\beta,2}\delta_{\gamma,0}\delta_{\tau,0}B,
    \end{align*}
    where $B=\{\varphi_\lambda\in\BDer\big(\hvsii\big)\ |\ \varphi_\lambda(L,L)=c(\partial+2\lambda)E,\ \forall\ c\in\bCx\}$. \textup{(}The terms that does not appear here are all zero\textup{)}.
\end{theorem}

One can obtain the following conclusions by referring to the proof of Lemmas~\ref{bider.hvs} and \ref{bider.hvsii}:
\begin{corollary}
    Let $R$ be a Lie conformal superalgebra.

    \subno{1} If $R$ is isomorphic to the Lie conformal superalgebra $R_4$ in Proposition~\ref{rank1.1}, then
    \begin{align*}
        \BDer(R)=\BInn(R)\oplus\delta_{\beta,2}\delta_{\gamma,0}B,
    \end{align*}
     where $B=\{\varphi_\lambda\in\BDer(R)\ |\ \varphi_\lambda(x,x)=c(\partial+2\lambda)y,\ \forall\ c\in\bCx\}$. \textup{(}The terms that does not appear here are all zero\textup{)}.

     \subno{2} If $R$ is isomorphic to the Lie conformal superalgebra $R_5$ in Proposition~\ref{rank1.1}, then
    \begin{align*}
        \BDer(R)=\BInn(R).
    \end{align*}
\end{corollary}

\section{Automorphism groups}\label{sect:automor}
\hspace{1.5em}In this section, the automorphism groups of $\hvs$ and $\hvsii$ are discussed and computed in detail.
\begin{definition}\label{homo}
    Let $R, R'$ be two Lie conformal superalgebras. A \textit{graded-preserved} linear map $\sigma:R\rightarrow R'$ is a \textit{homomorphism} of Lie conformal superalgebras if $\sigma$ satisfies $\sigma\partial=\partial\sigma$ and
    \begin{align*}
        \sigma([a_\lambda b])=[\sigma(a)_\lambda\sigma(b)],
    \end{align*}
    for any $a,b\in R$. For the Lie conformal superalgebra $R$, the set of all homomorphisms $R\rightarrow R$ is called \textit{automorphism group} under the composition of homomorphisms and denoted by $\Aut(R)$.
\end{definition}
\begin{remark}
    The condition graded-preserved in Definition~\ref{homo} means that $\sigma(R_\theta)=R_\theta$, $\theta\in\bZz$. In other words, we can say that these linear maps are homogeneous of degree zero(see \cite{M}).
\end{remark}

For $\sigma\in\Aut(\hvs)$, we have $\sigma(\hvs_\ep)=\hvs_\ep=\hv$, namely $\sigma$ is an automorphism of $\hv$. Since $\hv=\cp L\oplus \cp H$, we can assume that
\begin{align*}
    \sigma(L)&=f_L(\partial)L+g_L(\partial)H,\\
    \sigma(H)&=f_H(\partial)L+g_H(\partial)H,
\end{align*}
where $f_L(\partial),g_L(\partial),f_H(\partial),g_H(\partial)\in\cp$. Firstly, applying $\sigma$ to $[H_\lambda H]=0$ and comparing the coefficients of $L$, we can get
\begin{align}\label{sigmaHH}
    (\partial+2\lambda)f_H(-\lambda)f_H(\partial+\lambda)=0.
\end{align}

Then, applying $\sigma$ to $[L_\lambda L]=(\partial+2\lambda)L$ and comparing the coefficients of $L$ and $H$ respectively, we can obtain that
\begin{align}
    f_L(\partial)&=f_L(-\lambda)f_L(\partial+\lambda),\label{sigmaLL.L}\\
    (\partial+2\lambda)g_L(\partial)&=(\partial+\lambda)f_L(-\lambda)g_L(\partial+\lambda)+\lambda g_L(-\lambda)f_L(\partial+\lambda).\label{sigmaLL.H}
\end{align}
Furthermore, applying $\sigma$ to $[L_\lambda H]=(\partial+\lambda)H$ and comparing the coefficients of $H$, we have
\begin{align}\label{sigmaLH.H}
    (\partial+\lambda)g_H(\partial)=(\partial+\lambda)f_L(-\lambda)g_H(\partial+\lambda)+\lambda g_L(-\lambda)f_H(\partial+\lambda).
\end{align}
\begin{lemma}\label{automor.hv.1}
    For $\sigma\in\Aut(\hv)$, we have $f_L(\partial)=1$, $f_H(\partial)=0$ and $g_H(\partial)=c^2$, where $c^2\in\bC$ is a constant.
\end{lemma}
\begin{proof}
    By \eqref{sigmaHH}, we can get $f_H(\partial)=0$ directly, namely $\sigma(H)=g_H(\partial)H$. Then $f_L(\partial)\neq 0$, otherwise $\sigma(L)=g_L(\partial)H$ and this is a contradiction since $L\notin\sigma(\hv)$. On the other hand, we have $f_L(\partial)=f_L(-\lambda)f_L(\partial+\lambda)$ by \eqref{sigmaLL.L}. Obviously, it forces that $f_L(\partial)=1$. These conditions together with \eqref{sigmaLH.H}, show that $g_H(\partial)=g_H(\partial+\lambda)$, which implies that $g_H(\partial)$ is a constant. We denote this constant by $c^2$ and this lemma holds.
\end{proof}
\begin{lemma}\label{automor.hv.2}
    For $\sigma\in\Aut(\hv)$, we have $g_L(\partial)=g_0+g_1\partial$, where $g_0,g_1\in\bC$.
\end{lemma}
\begin{proof}
    By Lemma~\ref{automor.hv.1} and \eqref{sigmaLL.H}, we have
    \begin{align}\label{sigmaLL.H.1}
        (\partial+2\lambda)g_L(\partial)=(\partial+\lambda)g_L(\partial+\lambda)+\lambda g_L(-\lambda).
    \end{align}
     Suppose that $g_L(\partial)=\sum_{i=0}^ng_i\partial^i$, where $g_i\in\bC$. If $n>1$, comparing the coefficients of $\partial^n$ on the both sides of \eqref{sigmaLL.H.1}, we get $(n-1)\lambda g_n=0$. It follows that $g_n=0$ for $n>1$. Therefore $g_L(\partial)=g_0+g_1\partial$.
\end{proof}
\begin{proposition}\label{prop:automor.hvs}
    Assume $c\in\bCx$, $g_0,g_1\in\bC$. Then $\Aut(\hvs)=\langle\sigma\rangle$, where $\sigma$ is the homomorphism such that$$\sigma: L\mapsto L+(g_0+g_1\partial)H,\ H\mapsto c^2H,\ G\mapsto cG.$$
\end{proposition}
\begin{proof}
    Due to Lemmas~\ref{automor.hv.1} and \ref{automor.hv.2}, we know that $\sigma(L)=L+(g_0+g_1\partial)H$ and $\sigma(H)=c^2H$. So we only need to show that $\sigma(G)=cG$. Assume that $\sigma(G)=h_G(\partial)G$, where $h_G(\partial)\in\cp$. Applying $\sigma$ to $[G_\lambda G]=2H$, we  obtain $2h_G(-\lambda)h_G(\partial+\lambda)=2c^2$. It follows that $h_G(\partial)=\pm c$, namely $\sigma(G)=\pm cG$. Furthermore  $c\neq 0$ must hold, otherwise $\sigma(\hvs_\op)=\hvs_\op=0$ and this is a contradiction. Note that for any $\bar{c}\in\bCx$ such that $\sigma(G)=\bar{c}G$, we have $\sigma(H)=\bar{c}^2H$. Thus we get the conclusion.
\end{proof}

For $\sigma\in\Aut\big(\hvsii\big)$, we also have $\sigma\big(\hvsii_\ep\big)=\hvsii_\ep=\hv$. Therefore Lemmas~\ref{automor.hv.1} and \ref{automor.hv.2} still hold for the case of $\hvsii$.
\begin{proposition}\label{prop:automor.hvsii}
    Assume $c,h,\bar{h}\in\bCx$, $g_0,g_1\in\bC$.

    \subno{1} If $\tau\neq 0$, then $\Aut\big(\hvsii\big)=\langle\varsigma_h\ |\ h\in\bCx\rangle$, where
    $\varsigma_h$ is the homomorphism such that $$\varsigma_h: L\mapsto L,\ H\mapsto H,\ E\mapsto hE.$$

    \subno{2} If $\tau=0$,  then $\Aut\big(\hvsii\big)=\langle\varsigma\rangle$, where
    $\varsigma$ is the homomorphism such that $$\varsigma: L\mapsto L+(g_0+g_1\partial)H,\ H\mapsto c^2H,\ E\mapsto \bar{h}E.$$
\end{proposition}
\begin{proof}
    By Lemmas~\ref{automor.hv.1} and \ref{automor.hv.2}, we know that $\sigma(L)=L+(g_0+g_1\partial)H$ and $\sigma(H)=c^2H$ for any $\sigma\in\Aut\big(\hvsii\big)$. Now, suppose $\sigma(E)=h_E(\partial)E$, where $h_E(\partial)\in\cp$. Then $h_E(\partial)\neq 0$, otherwise $\sigma\big(\hvsii_\op\big)=\hvsii_\op=0$ and this is a contradiction. Applying $\sigma$ to $[H_\lambda E]=\tau E$, we obtain 
    \begin{align}\label{sigmaHE}
        c^2\tau h_E(\partial+\lambda)=\tau h_E(\partial).
    \end{align}

    \subno{1} If $\tau\neq 0$, by \eqref{sigmaHE}, we can assume that $h_E(\partial)=h\in\bCx$. This together with \eqref{sigmaHE}, shows that $(c^2-1)h=0$. Thus $c^2=1$. Now applying $\sigma$ to $[L_\lambda E]=(\partial+\beta\lambda+\gamma)E$, we have
    \begin{align}\label{sigmaLE}
        (\partial+\beta\lambda+\gamma)h_E(\partial+\lambda)+\tau(g_0-g_1\lambda)h_E(\partial+\lambda)=(\partial+\beta\lambda+\gamma)h_E(\partial).
    \end{align}
    Putting $h_E(\partial)=h$ in \eqref{sigmaLE}, we get $\tau(g_0-g_1\lambda)h=0$. It follows that $g_0=g_1=0$. This automorphism depends on $h$ and we denote it by $\varsigma_h$, then we have the conclusion.

    \subno{2} If $\tau=0$, applying $\sigma$ to $[L_\lambda E]=(\partial+\beta\lambda+\gamma)E$, then $h_E(\partial+\lambda)=h_E(\partial)$. Therefore $h_E(\partial)$ is a nonzero constant which we denoted by $\bar{h}$, namely $\sigma(E)=\bar{h}E$. Denote this automorphism by $\varsigma$, then this proposition holds .
\end{proof}

Furthermore, we have the following theorem.
\begin{theorem}\label{thm:automor}
    The following statements hold:

    \subno{1} For the Heisenberg--Virasoro Lie conformal superalgebra \hvs, we have
    \begin{align*}
        \Aut(\hvs)\cong\bCx\ltimes(\bC\times\bC).
    \end{align*}

    \subno{2} For the Lie conformal superalgebra \hvsii, we have
    \begin{align*}
        \Aut\big(\hvsii\big)\cong\begin{cases}
                \bCx &\text{if } \tau\neq 0,\\
                 (\bCx\times\bCx)\ltimes(\bC\times\bC) &\text{if } \tau=0.
            \end{cases}
    \end{align*}
\end{theorem}
\begin{proof}
    Assume $c,h\in\bCx$, $\eta\in\bC$.

    \subno{1} Let $\sigma$ be an automorphism of $\Aut(\hvs)$. Define the following homomorphisms:
    \begin{align*}
        \rho_c:\ &L\mapsto L,\ H\mapsto c^2H,\ G\mapsto cG,\\
        \phi_{\eta}:\ &L\mapsto L+\eta H,\ H\mapsto H,\ G\mapsto G,\\
        \pi_{\eta}:\ &L\mapsto L+\eta\partial H,\ H\mapsto H,\ G\mapsto G.
    \end{align*}
    Firstly, one can directly verify that the set $\{\rho_c\ |\ c\in\bCx\}\cong\bCx$ is a subgroup of $\Aut(\hvs)$ under $\rho_{c_1}\rho_{c_2}=\rho_{c_1c_2}$, where $c_1, c_2\in\bCx$. Similarly, the set $\{X_\eta\ |\ \eta\in\bC\}\cong\bC$ is a subgroup of $\Aut(\hvs)$ under $X_{\eta_1}X_{\eta_2}=X_{\eta_1+\eta_2}$, where $X\in\{\phi, \pi\}$, $\eta_1, \eta_2\in\bC$. In addition, we have $\rho_c\phi_\eta\rho_c^{-1}=\phi_{c^2\eta}$ and $\rho_c\pi_\eta\rho_c^{-1}=\pi_{c^2\eta}$.
    
    Now we denote the automorphism of $\Aut(\hvs)$ by $\sigma(c,g_0,g_1)$, by Proposition~\ref{prop:automor.hvs}, we can obtain that
    \begin{align*}
      \sigma(c,g_0,g_1)\sigma(c',g'_0,g'_1)=\sigma(cc',g_0+c^2g'_0,g_1+c^2g'_1).
    \end{align*}
    And $\sigma(c,g_0,g_1)=\sigma(c',g'_0,g'_1)$ if and only if $c=c'$, $g_0=g'_0$, $g_1=g'_1$. Define the following homomorphisms:
    \begin{align*}
      \sigma(c,g_0,g_1)\mapsto\sigma(c,0,0),\ \sigma(c,g_0,g_1)\mapsto\sigma(1,g_0,0),\ \sigma(c,g_0,g_1)\mapsto\sigma(1,0,g_1).
    \end{align*} 
    It is clearly that $\sigma(c,0,0), \sigma(1,g_0,0), \sigma(1,0,g_1)$ are isomorphic to $\rho_c, \phi_{g_0}, \pi_{g_1}$, respectively. Thus $\Aut(\hvs)\cong\bCx\ltimes(\bC\times\bC)$.

    \subno{2} If $\tau\neq 0$, by Proposition~\ref{prop:automor.hvsii}(1), one can directly prove that the set $\{\varsigma_h\ |\ h\in\bCx\}\cong\bCx$ is the automorphism group of $\Aut\big(\hvsii\big)$ under $\varsigma_{h_1}\varsigma_{h_2}=\varsigma_{h_1h_2}$, where $h_1, h_2\in\bCx$. So we can get $\Aut\big(\hvsii\big)\cong\bCx$.

    If $\tau=0$, let $\bar{\varsigma}$ be an automorphism of $\Aut\big(\hvsii\big)$. Define the following homomorphisms:
    \begin{align*}
        \varrho_c:\ &L\mapsto L,\ H\mapsto c^2H,\ E\mapsto E,\\
        \vartheta_c:\ &L\mapsto L,\ H\mapsto H,\ E\mapsto cE,\\
        \psi_{\eta}:\ &L\mapsto L+\eta H,\ H\mapsto H,\ E\mapsto E,\\
        \varpi_{\eta}:\ &L\mapsto L+\eta\partial H,\ H\mapsto H,\ E\mapsto E.
    \end{align*}
    Similarly, the set $\{X'_c\ |\ c\in\bCx\}\cong\bCx$ is a subgroup of $\Aut\big(\hvsii\big)$ under $X'_{c_1}X'_{c_2}=X'_{c_1c_2}$, where $X'\in\{\varrho_c, \vartheta_c\}$, $c_1,c_2\in\bCx$. And the set $\{X''_\eta\ |\ \eta\in\bC\}\cong\bC$ is a subgroup of  $\Aut\big(\hvsii\big)$ under $X''_{\eta_1}X''_{\eta_2}=X''_{\eta_1+\eta_2}$, where $X''\in\{\psi_\eta, \varpi_\eta\}$, $\eta_1, \eta_2\in\bC$. In addition, we have
    \begin{align*}
    \varrho_c\psi_{\eta}\varrho_c^{-1}=\psi_{c^2\eta},\ \vartheta_c\psi_{\eta}\vartheta_c^{-1}=\psi_{c\eta},\
    \varrho_c\varpi_{\eta}\varrho_c^{-1}=\varpi_{c^2\eta},\
    \vartheta_c\varpi_{\eta}\vartheta_c^{-1}=\varpi_{c\eta}.
    \end{align*}
    Similarly, together with Proposition~\ref{prop:automor.hvsii}(2), we can get   $\Aut\big(\hvsii\big)\cong(\bCx\times\bCx)\ltimes(\bC\times\bC)$.
\end{proof}

\section{Conformal modules of rank $(1+1)$ }\label{sect:conformalmodules}
\hspace{1.5em}In this section, we aim to classify the free conformal modules of rank $(1+1)$ over $\hvs$ and $\hvsii$, respectively. We only concentrate on the conformal modules with nontrivial $\hvs_\op$-action(resp., $\hvsii_\op$-action), otherwise the modules degenerate to the Heisenberg--Virasoro conformal modules(see Proposition~\ref{hv.mod}).

Before giving the main results of this section, we introduce two technical results about polynomial equations in \cite{WWX,X22}. Let $a, b, c\in\bC$ and $f(x,y)\in\bC[x,y]$. Consider the following polynomial equation on $f(x,y)$:
\begin{align}\label{key.equ}
    (x+by)f(x+y,z)-(x+ay+z)f(x,z)=(cy-z)f(x,y+z).
\end{align}
\begin{lemma}\label{lemma.equ.1}\textup{(\cite{X22})}
    The equation \eqref{key.equ} with $c\neq 0$ has nontrivial solutions only if $a-b+c=0,1,2,3$. The following table is a complete list of all possible values of $a,b,c$ along with $f(x,y)$, whose nonzero scalar multiples give rise to nontrivial solutions.
    \[
    \begin{tabular}{cccc}
    \toprule
    \phantom{1234}$a-b+c$\phantom{1234} & \phantom{123456}$b$\phantom{123456} & \phantom{123456}$c$\phantom{123456} & \phantom{123}$f(x,y)$\phantom{123}\\
    \midrule
    $0$ & $b$ & $c$ & $1$\\
    $1$ & $b$ & $c$ & $x+\frac{b}{c}y$\\
    $2$ & $c-1$ & $c\neq 1$ & $(x+y)+\big(x+(1-\frac{1}{c})y\big)$\\
    $2$ & $0$ & $c$ & $x(x+\frac{1}{c}y)$\\
    $3$ & $-\frac{3}{2}$ & $\frac{2}{3}$ & $(x-y)(x+\frac{1}{2}y)(x+2y)$\\
    $3$ & $0$ & $2$ & $x(x+\frac{1}{2}y)(x+y)$\\
    \bottomrule
    \end{tabular}
    \]
\end{lemma}

For the case of $c=0$, one solution was missed in the original conclusion in \cite{X22} and the conclusion was revised in \cite{WWX}.
\begin{lemma}\label{lemma.equ.2}\textup{(\cite{WWX})}
    The equation \eqref{key.equ} with $c=0$ has nontrivial solutions only if $a-b=0,1,2,3$. The following table is a complete list of all possible values of $a,b$ along with $f(x,y)$, whose nonzero scalar multiples give rise to nontrivial solutions.
    \[
    \begin{tabular}{ccc}
    \toprule
    \phantom{12345}$a-b$\phantom{12345} & \phantom{1234567}$b$\phantom{1234567} & \phantom{123456}$f(x,y)$\phantom{123456}\\
    \midrule
    $0$ & $b$ & $1$\\
    $1$ & $b$ & $y$\\
    $1$ & $0$ & $x+ky,k\in\bC$\\
    $2$ & $b$ & $y(x-by)$\\
    $3$ & $-2$ & $y(x+y)(x+2y)$\\
    $3$ & $0$ & $yx(x-y)$\\
    \bottomrule
    \end{tabular}
    \]
\end{lemma}

Now we classify the free conformal modules of rank $(1+1)$ over $\hvs$ and $\hvsii$, respectively. As a matter of fact, the finite irreducible nontrivial conformal modules over $\hvs$ have been determined completely as a special subalgebra of a class of twisted Lie conformal superalgebra of infinite rank in \cite{X22}. However, the explicit $\lambda$-actions were not given. We present them here with calculation process.
\begin{theorem}\label{thm:conmod.hvs}
    Up to parity change, any nontrivial free conformal $\hvs$-module of rank $(1+1)$ is isomorphic to one of the followings\textup{:}

    \subno{1} $M_{\Delta,a,b,c}=\cp v_\ep\oplus\cp v_\op$ with $\lambda$-actions of \hvs\textup{:}
    \begin{align*}
       \begin{cases}
           L_\lambda v_\ep=(\partial+\Delta\lambda+a)v_\ep,\ L_\lambda v_\op=(\partial+\Delta\lambda+a)v_\op,\\
           H_\lambda v_\ep=bcv_\ep,\ H_\lambda v_\op=bcv_\op,\\
           G_\lambda v_\ep=bv_\op,\phantom{c}\ G_\lambda v_\op=cv_\ep,
       \end{cases}
    \end{align*}
    where $\Delta,a\in\bC$, $b,c\in\bCx$.

    \subno{2} $M_{\Delta_0,\Delta_1,a}^{(i)}=\cp v_\ep\oplus\cp v_\op\ (i=1,2,\dots,6)$ with $\lambda$-actions of \hvs\textup{:}
    \begin{align*}
        \begin{cases}
            L_\lambda v_\ep=(\partial+\Delta_0\lambda+a)v_\ep,\ L_\lambda v_\op=(\partial+\Delta_1\lambda+a)v_\op,\\
           H_\lambda v_\ep=H_\lambda v_\op=0,\\
           G_\lambda v_\ep=h(\partial,\lambda)v_\op,\ G_\lambda v_\op=0,
        \end{cases}
    \end{align*}
    where $\Delta_0,\Delta_1,a\in\bC$ and $h(\partial,\lambda)\in\bC[\partial,\lambda]$ satisfying the following relations\textup{:}
    \[
    \begin{tabular}{cccc}
    \toprule
    \phantom{12}\text{conformal module}\phantom{12} & \phantom{1234}$\Delta_0$\phantom{1234} & \phantom{1234}$\Delta_1$\phantom{1234} & \phantom{12}$h(\partial,\lambda)$\phantom{12}\\
    \midrule
    $M_{\Delta_0,\Delta_1,a}^{(1)}$ & $\Delta_0$ & $\Delta_0$ & $1$\\
    $M_{\Delta_0,\Delta_1,a}^{(2)}$ & $\Delta_0$ & $\Delta_0-1$ & $\lambda$\\
    $M_{\Delta_0,\Delta_1,a}^{(3)}$ & $\Delta_0$ & $\Delta_0-2$ & $\lambda(\partial-\Delta_1\lambda+a)$\\
    $M_{\Delta_0,\Delta_1,a}^{(4)}$ & $1$ & $0$ & $\partial+k\lambda+a,k\in\bC$\\
    $M_{\Delta_0,\Delta_1,a}^{(5)}$ & $1$ & $-2$ & $\lambda(\partial+\lambda+a)(\partial+2\lambda+a)$\\
    $M_{\Delta_0,\Delta_1,a}^{(6)}$ & $3$ & $0$ & $\lambda(\partial+a)(\partial-\lambda+a)$\\
    \bottomrule
    \end{tabular}
    \]
\end{theorem}
\begin{proof}
    Let $M=\cp v_\ep\oplus\cp v_\op$ be a nontrivial free conformal \hvs-module of rank $(1+1)$ with nontrivial $\hvs_\op$-action. Assume that
    \begin{align*}
       \begin{cases}
           L_\lambda v_\ep=f_0(\partial,\lambda)v_\ep,\ L_\lambda v_\op=f_1(\partial,\lambda)v_\op,\\
           H_\lambda v_\ep=g_0(\partial,\lambda)v_\ep,\ H_\lambda v_\op=g_1(\partial,\lambda)v_\op,\\
           G_\lambda v_\ep=h_0(\partial,\lambda)v_\op,\ G_\lambda v_\op=h_1(\partial,\lambda)v_\ep,\\
       \end{cases}
    \end{align*}
    where $f_i(\partial,\lambda),g_i(\partial,\lambda),h_i(\partial,\lambda)(i=0,1)\in\bC[\partial,\lambda]$ and $\big(h_0(\partial,\lambda),h_1(\partial,\lambda)\big)\neq(0,0)$. Up to parity change, we may further assume that $h_0(\partial,\lambda)\neq 0$. From $L_\lambda(G_\mu v_\ep)-G_\mu(L_\lambda v_\ep)=[L_\lambda G]_{\lambda+\mu}v_\ep$, we have
    \begin{align}\label{LGv0}
        f_1(\partial,\lambda)h_0(\partial+\lambda,\mu)-f_0(\partial+\mu,\lambda)h_0(\partial,\mu)=-\mu h_0(\partial,\lambda+\mu).
    \end{align}
    It forces that $f_0(\partial,\lambda)\neq 0$. Otherwise, taking $f_0(\partial,\lambda)=0$ and $\lambda=0$ in \eqref{LGv0}, we obtain $\big(f_1(\partial,0)+\mu\big)h_0(\partial,\mu)=0$, which implies that $h_0(\partial,\mu)=0$. This is a contradiction with our assumption. Since $M_\ep(=\cp v_\ep)$ is the conformal module over $\hv$, by Proposition~\ref{hv.mod}, we can assume $f_0(\partial,\lambda)=\partial+\Delta_0\lambda+a_0$ for some $\Delta_0,a_0\in\bC$. In a similar way, we can deduce that $f_1(\partial,\lambda)=\partial+\Delta_1\lambda+a_1$ for some $\Delta_1,a_1\in\bC$. Now, the equality \eqref{LGv0} becomes $$(\partial+\mu+\Delta_0\lambda+a_0)h_0(\partial,\mu)-(\partial+\Delta_1\lambda+a_1)h_0(\partial+\lambda,\mu)=\mu h_0(\partial,\lambda+\mu).$$ Putting $\lambda=0$ in the above equality, we get $(a_0-a_1)h_0(\partial,\mu)=0$, which implies $a_0=a_1$. We redenote $a_0$ by $a$ and obtain
    \begin{align}\label{LGv0.1}
        (\partial+\mu+\Delta_0\lambda+a)h_0(\partial,\mu)-(\partial+\Delta_1\lambda+a)h_0(\partial+\lambda,\mu)=\mu h_0(\partial,\lambda+\mu).
    \end{align}

    From $[G_\lambda G]=2H$ and \eqref{conmod.2}, it follows that $G_\lambda(G_\mu v_\ep)+G_\mu(G_\lambda v_\ep)=2H_{\lambda+\mu}v_\ep$. Thus we can obtain
    \begin{align}\label{GGv0}
        h_0(\partial+\lambda,\mu)h_1(\partial,\lambda)+h_0(\partial+\mu,\lambda)h_1(\partial,\mu)=2g_0(\partial,\lambda+\mu).
    \end{align}
    Putting $\mu=\lambda$ in \eqref{GGv0}, gives that
    \begin{align}\label{GGv0.1}
        g_0(\partial,2\lambda)=h_0(\partial+\lambda,\lambda)h_1(\partial,\lambda).
    \end{align}
    Similarly, from $G_\lambda(G_\mu v_\op)+G_\mu(G_\lambda v_\op)=2H_{\lambda+\mu}v_\op$, we can get
    \begin{align}\label{GGv1}
        g_1(\partial,2\lambda)=h_0(\partial,\lambda)h_1(\partial+\lambda,\lambda).
    \end{align}
    Given $M_\ep$ is a conformal module over $\hv$, by Proposition~\ref{hv.mod}, we know that $g_0(\partial,\lambda)$ is a constant. Thus both $h_0(\partial,\lambda)$ and $h_1(\partial,\lambda)$ are constants by \eqref{GGv0.1}.

    \afCase{1} $g_0(\partial,\lambda)\neq 0$.

    We can deduce that $h_0(\partial,\lambda)=b$ and $h_1(\partial,\lambda)=c$, where $b,c\in\bC$ and $bc\neq 0$. By \eqref{GGv0.1} and \eqref{GGv1}, we can derive that $g_0(\partial,\lambda)=g_1(\partial,\lambda)=bc$. Then \eqref{LGv0.1} turns into $b(\Delta_1-\Delta_0)\lambda=0$. It follows that $\Delta_0=\Delta_1$. Redenote $\Delta_0$ by $\Delta$. We get the conformal module $M_{\Delta,a,b,c}$ and its $\lambda$-actions.

    \afCase{2} $g_0(\partial,\lambda)=0$.

    We have $h_0(\partial+\lambda,\lambda)h_1(\partial,\lambda)=0$ by \eqref{GGv0.1}. It leads to $h_1(\partial,\lambda)=0$ since $h_0(\partial,\lambda)\neq 0$ under our assumption. Furthermore, we get $g_1(\partial,\lambda)=0$ via \eqref{GGv1}. Let $h(\partial+a,\lambda)=h_0(\partial,\lambda)$ and $\bp=\partial+a$. The equality \eqref{LGv0.1} becomes
    \begin{align*}
        (\bp+\Delta_0\lambda+\mu)h(\bp,\mu)-(\bp+\Delta_1\lambda)h(\bp+\lambda,\mu)=\mu h(\bp,\lambda+\mu).
    \end{align*}
    In fact, the nontrivial solutions of this equation (up to scalar multiples) have been shown in Lemma~\ref{lemma.equ.2}. Therefore we get the conformal modules $M_{\Delta_0,\Delta_1,a}^{(i)}\ (i=1,2,\dots,6)$ and their $\lambda$-actions.
    \end{proof}

    \begin{theorem}\label{thm:conmod.hvsii}
    Up to parity change, any nontrivial free conformal $\hvs(\beta,\gamma,0)$-module of rank $(1+1)$ is isomorphic to one of the followings\textup{:}

    \subno{1} $M_{\Delta_0,\Delta_1,a}^{(i)}=\cp v_\ep\oplus\cp v_\op\ (i=1,2,\dots,10)$ with $\lambda$-actions of $\hvs(\beta,\gamma,0)$\textup{:}
    \begin{align*}
        \begin{cases}
            L_\lambda v_\ep=(\partial+\Delta_0\lambda+a-\gamma)v_\ep,\ L_\lambda v_\op=(\partial+\Delta_1\lambda+a)v_\op,\\
           H_\lambda v_\ep=H_\lambda v_\op=0,\\
           E_\lambda v_\ep=h(\partial,\lambda)v_\op,\ E_\lambda v_\op=0,
        \end{cases}
    \end{align*}
    where $\Delta_0,\Delta_1,a\in\bC$ and $h(\partial,\lambda)\in\bC[\partial,\lambda]$,  $\bp=\partial+a$ and $\bl=\lambda-\gamma$, satisfying the following relations\textup{:}
    \[
    \begin{tabular}{ccccc}
    \toprule
    \text{conformal module} & $\beta$ & $\Delta_0$ & $\Delta_1$ & $h(\partial,\lambda)$\\
    \midrule
    $M_{\Delta_0,\Delta_1,a}^{(1)}$ & $\beta$ & $\Delta_0$ & $\Delta_0+\beta-1$ & $1$\\
    $M_{\Delta_0,\Delta_1,a}^{(2)}$ & $1$ & $\Delta_0$ & $\Delta_0-1$ & $\bl$\\
    $M_{\Delta_0,\Delta_1,a}^{(3)}$ & $1$ & $\Delta_0$ & $\Delta_0-2$ & $\bl(\bp-\Delta_1\bl)$\\
    $M_{\Delta_0,\Delta_1,a}^{(4)}$ & $\beta$ & $k(\beta-1)-\beta+2$ & $k(\beta-1)$ & $\bp+k\bl,k\in\bC$\\
    $M_{\Delta_0,\Delta_1,a}^{(5)}$ & $\beta\neq 1,2$ & $1$ & $\beta-2$ & $(\bp+\bl)(\partial+\frac{\beta-2}{\beta-1}\bl)$\\
    $M_{\Delta_0,\Delta_1,a}^{(6)}$ & $\beta\neq 1$ & $3-\beta$ & $0$ & $\bp(\bp+\frac{1}{\beta-1}\bl)$\\
    $M_{\Delta_0,\Delta_1,a}^{(7)}$ & $\frac{5}{3}$ & $\frac{5}{6}$ & $-\frac{3}{2}$ & $(\bp-\bl)(\bp+\frac{1}{2}\bl)(\bp+2\bl)$\\
    $M_{\Delta_0,\Delta_1,a}^{(8)}$ & $3$ & $1$ & $0$ & $\bp(\bp+\frac{1}{2}\bl)(\bp+\bl)$\\
    $M_{\Delta_0,\Delta_1,a}^{(9)}$ & $1$ & $1$ & $-2$ & $\bl(\bp+\bl)(\bp+2\bl)$\\
    $M_{\Delta_0,\Delta_1,a}^{(10)}$ & $1$ & $3$ & $0$ & $\bl\bp(\bp-\bl)$\\
    \bottomrule
    \end{tabular}
    \]

    \subno{2} $M_{\Delta_0,\Delta_1,a,b}^{(i)}=\cp v_\ep\oplus\cp v_\op\ (i=1,2)$ with $\lambda$-actions of $\hvs(\beta,\gamma,0)$\textup{:}
    \begin{align*}
        \begin{cases}
            L_\lambda v_\ep=(\partial+\Delta_0\lambda+a-\gamma)v_\ep,\ L_\lambda v_\op=(\partial+\Delta_1\lambda+a)v_\op,\\
           H_\lambda v_\ep=bv_0,\ H_\lambda v_\op=bv_\op,\\
           E_\lambda v_\ep=h(\partial,\lambda)v_\op,\ E_\lambda v_\op=0,
        \end{cases}
    \end{align*}
    where $\Delta_0,\Delta_1,a\in\bC$, $b\in\bCx$ and $h(\partial,\lambda)\in\bC[\partial,\lambda]$ satisfying the following relations\textup{:}
    \[
    \begin{tabular}{ccccc}
    \toprule
    \phantom{12}\text{conformal module}\phantom{12} & \phantom{1234}$\beta$\phantom{1234} & \phantom{123}$\Delta_0$\phantom{123} & \phantom{123}$\Delta_1$\phantom{123} & \phantom{12}$h(\partial,\lambda)$\phantom{12}\\
    \midrule
    $M_{\Delta_0,\Delta_1,a,b}^{(1)}$ & $\beta$ & $\Delta_0$ & $\Delta_0+\beta-1$ & $1$\\
    $M_{\Delta_0,\Delta_1,a,b}^{(2)}$ & $1$ & $\Delta_0$ & $\Delta_0-1$ & $\lambda-\gamma$\\
    \bottomrule
    \end{tabular}
    \]
\end{theorem}
\begin{proof}
    Similar to the discussion in Theorem~\ref{thm:conmod.hvs}, we can assume $M=\cp v_\ep\oplus\cp v_\op$ is a nontrivial free conformal \hvsii-module of rank $(1+1)$ with nontrivial $\hvsii_\op$-action and
    \begin{align*}
       \begin{cases}
           L_\lambda v_\ep=f_0(\partial,\lambda)v_\ep,\ L_\lambda v_\op=f_1(\partial,\lambda)v_\op,\\
           H_\lambda v_\ep=g_0(\partial,\lambda)v_\ep,\ H_\lambda v_\op=g_1(\partial,\lambda)v_\op,\\
           E_\lambda v_\ep=h_0(\partial,\lambda)v_\op,\ E_\lambda v_\op=h_1(\partial,\lambda)v_\ep,\\
       \end{cases}
    \end{align*}
    where $f_i(\partial,\lambda),g_i(\partial,\lambda),h_i(\partial,\lambda)(i=0,1)\in\bC[\partial,\lambda]$ and $\big(h_0(\partial,\lambda),h_1(\partial,\lambda)\big)\neq(0,0)$. Up to parity change, we may further assume that $h_0(\partial,\lambda)\neq 0$. From $L_\lambda(E_\mu v_\ep)-E_\mu(L_\lambda v_\ep)=[L_\lambda E]_{\lambda+\mu}v_\ep$, we have
    \begin{align}\label{LEv0}
        f_1(\partial,\lambda)h_0(\partial+\lambda,\mu)-f_0(\partial+\mu,\lambda)h_0(\partial,\mu)=\big((\beta-1)\lambda-\mu+\gamma\big)h_0(\partial,\lambda+\mu).
    \end{align}
    It follows that $f_0(\partial,\lambda)\neq 0$. Otherwise, taking $f_0(\partial,\lambda)=0$ and $\lambda=0$ in \eqref{LEv0}, we obtain $\big(f_1(\partial,0)+\mu-\gamma\big)h_0(\partial,\mu)=0$, which implies that $h_0(\partial,\mu)=0$. This contradicts with our assumption. Since $M_\ep(=\cp v_\ep)$ and $M_\op(=\cp v_\op)$ are the conformal modules over $\hv$, by Proposition~\ref{hv.mod}, we can assume $f_0(\partial,\lambda)=\partial+\Delta_0\lambda+a_0$ and $f_1(\partial,\lambda)=\partial+\Delta_1\lambda+a_1$ for some $\Delta_0,\Delta_1,a_0,a_1\in\bC$. Now, the equality \eqref{LEv0} becomes $$(\partial+\Delta_1\lambda+a_1)h_0(\partial+\lambda,\mu)-(\partial+\mu+\Delta_0\lambda+a_0)h_0(\partial,\mu)=\big((\beta-1)\lambda-\mu+\gamma\big)h_0(\partial,\lambda+\mu).$$ Putting $\lambda=0$ in the above equality, we get $(a_1-a_0-\gamma)h_0(\partial,\mu)=0$, which implies $a_0=a_1-\gamma$. For convenience, we redenote $a_1$ by $a$ and let $\bp=\partial+a$, $\bl=\lambda-\gamma$ and $\bar{h}(\partial,\lambda)=h_0(\partial-a,\lambda+\gamma)$. Then $\bar{h}(\bp,\bl)=h_0(\partial,\lambda)$. Further let $\bar{\mu}=\mu-\gamma$. Finally, we can rewrite \eqref{LEv0} as
    \begin{align}\label{LEv0.2}
        (\bp+\Delta_1\bl)\bar{h}(\bp+\bl,\bar{\mu})-(\bp+\bar{\mu}+\Delta_0\bl)\bar{h}(\bp,\bar{\mu})=\big((\beta-1)\bl-\bar{\mu}\big)\bar{h}(\bp,\bl+\bar{\mu}).
    \end{align}
    Applying Lemma~\ref{lemma.equ.1} (if $\beta\neq 1$) and Lemma~\ref{lemma.equ.2} (if $\beta=1$), we can obtain all the nontrivial solutions (up to scalar multiples) of $\bar{h}(\bp,\bl)$($=h_0(\partial,\lambda)$).

    From $E_\lambda(E_\mu v_\ep)+E_\mu(E_\lambda v_\ep)=0$, it follows that
    \begin{align}\label{EEv0}
        h_0(\partial+\lambda,\mu)h_1(\partial,\lambda)+h_0(\partial+\mu,\lambda)h_1(\partial,\mu)=0.
    \end{align}
    Putting $\mu=\lambda$ in \eqref{EEv0}, we have $ h_0(\partial+\lambda,\lambda)h_1(\partial,\lambda)=0$. We can get $h_1(\partial,\lambda)=0$ since $h_0(\partial,\lambda)\neq 0$ under our assumption. Considering $H_\lambda(E_\mu v_\ep)-E_\mu(H_\lambda v_\ep)=[H_\lambda E]_{\lambda+\mu}v_\ep$, we obtain
    \begin{align}\label{HEv0}
        g_1(\partial,\lambda)h_0(\partial+\lambda,\mu)-g_0(\partial+\mu,\lambda)h_0(\partial,\mu)=\tau h_0(\partial,\lambda+\mu).
    \end{align}
  Given $M_\ep$ is a conformal module over $\hv$, by Proposition~\ref{hv.mod}, we have $g_0(\partial,\lambda)=b_0$ for some $b_0\in\bC$. Similarly, we also have $g_1(\partial,\lambda)=b_1$ for some $b_1\in\bC$. Furthermore, by \eqref{HEv0} with $\lambda=0$, we can deduce that $(b_1-b_0-\tau)h_0(\partial,\mu)=0$, which means that $b_0=b_1-\tau$. Redenote $b_1$ by $b$ and we get
  \begin{align}\label{HEv0.1}
      (b-\tau)h_0(\partial+\lambda,\mu)-bh_0(\partial,\mu)=\tau h_0(\partial,\lambda+\mu).
  \end{align}

  Firstly, we claim that $\tau=0$. Otherwise, suppose $\tau\neq 0$. Putting $\lambda=\mu=0$ in \eqref{HEv0.1}, then we obtain $\tau h_0(\partial,0)=0$. It leads to $h_0(\partial,0)=0$. Then taking $\mu=0$ in \eqref{HEv0.1}, we have $(b-\tau)h_0(\partial+\lambda,0)-bh_0(\partial,0)=\tau h_0(\partial,\lambda)$, namely $\tau h_0(\partial,\lambda)=0$. It forces that $ h_0(\partial,\lambda)=0$, which contradicts with our assumption. Thus $\tau=0$, and we can deduce that $b\big(h_0(\partial+\lambda,\mu)-h_0(\partial,\mu)\big)=0$ by \eqref{HEv0.1}. Furthermore, if $b=0$, namely the equality \eqref{HEv0.1} always holds, then all the solutions of $h_0(\partial,\lambda)$ in \eqref{LEv0.2} contribute to the conformal modules. Therefore we get the conformal modules $M_{\Delta_0,\Delta_1,a}^{(i)}\ (i=1,2,\dots,10)$ and their $\lambda$-actions. If $b\neq 0$, then $h_0(\partial+\lambda,\mu)=h_0(\partial,\mu)$, which implies that the value of $h_0(\partial,\lambda)$ is independent of $\partial$. By Lemmas~\ref{lemma.equ.1} and \ref{lemma.equ.2}, we obtain the conformal modules  $M_{\Delta_0,\Delta_1,a,b}^{(i)}\ (i=1,2)$ and their $\lambda$-actions.
\end{proof}

\begin{remark}
    According to the proof of Theorem~\ref{thm:conmod.hvsii}, we can obtain that, for the case of $\tau\neq 0$, $\hvsii$ has no conformal modules with nontrivial $\hvsii_\op$-action. In other words, its modules degenerate to the Heisenberg--Virasoro conformal modules (see Proposition~\ref{hv.mod}).
\end{remark}

\end{document}